\documentclass[a4paper]{amsart}

\usepackage{hyperref}
\usepackage{url}
\usepackage{amssymb,latexsym}
\usepackage[utf8x]{inputenc}
\usepackage{amsmath,amsthm}
\usepackage{amsfonts,mathrsfs}
\usepackage[capitalize]{cleveref}
\usepackage{enumerate,units}
\usepackage[all]{xy}
\usepackage{graphicx}
\usepackage[dvipsnames]{xcolor}
\usepackage{url}
\usepackage{tikz-cd}
\usepackage[title]{appendix}
\usepackage{etoolbox}

\newtoggle{THESIS}
\togglefalse{THESIS}

\theoremstyle{plain}
\newtheorem{theorem}{Theorem}[section]

\newtheorem{lemma}[theorem]{Lemma}
\newtheorem{proposition}[theorem]{Proposition}
\newtheorem{remarkcnt}[theorem]{Remark}
\newtheorem{notation}[theorem]{Notation}
\newtheorem{fact}[theorem]{Fact}
\newtheorem*{claim}{Claim}
\newtheorem*{theorem*}{Theorem}
\newtheorem{question}[theorem]{Question}

\theoremstyle{definition}
\newtheorem{definition}[theorem]{Definition}
\newtheorem{example}[theorem]{Example}

\theoremstyle{remark}

\newcommand{\Z}{\mathbb{Z}}
\newcommand{\Q}{\mathbb{Q}}

\newcommand{\F}{\mathbb{F}}

\newcommand{\set}[1]{\left\{ {#1} \right\}}
\newcommand{\vect}[1]{\langle {#1} \rangle}
\newcommand{\abs}[1]{\lvert {#1} \rvert}

\newcommand{\ACF}{\mathrm{ACF}}

\newcommand{\tp}{\mathrm{tp}}

\newcommand{\M}{\mathbb{M}}

\newcommand{\Mat}{\mathrm{Mat}}
\newcommand{\fr}{\mathrm{Frac}}

\newcommand{\EC}{\mathcal{EC}}

\newcommand{\forkindep}[1][]{%
  \mathrel{
    \mathop{
      \vcenter{
        \hbox{\oalign{\noalign{\kern-.3ex}\hfil$\vert$\hfil\cr
              \noalign{\kern-.7ex}
              $\smile$\cr\noalign{\kern-.3ex}}}
      }
    }\displaylimits_{#1}
  }
}

\title{Existentially closed models of fields with a distinguished submodule}

\author{Christian d\textquoteright Elb\'ee}
\address{Mathematisches Institut der Universität Bonn, Endenicher Allee 60, room 4.004, 
53115 Bonn, Germany}
\email{cdelbee@math.uni-bonn.de}
\urladdr{\href{http://choum.net/\textasciitilde chris/page\textunderscore perso/}{http://choum.net/\textasciitilde chris/page\textunderscore perso/}}

\author{Itay Kaplan}
\address{Einstein Institute of Mathematics, Hebrew University of
Jerusalem, 91904, Jerusalem Israel.}
\email{kaplan@math.huji.ac.il}
\urladdr{\href{http://math.huji.ac.il/~kaplan}{math.huji.ac.il/~kaplan}}

\author{Leor Neuhauser}
\address{Einstein Institute of Mathematics, Hebrew University of
Jerusalem, 91904, Jerusalem Israel.}
\email{leor.neuhauser@math.huji.ac.il}

\thanks{
The authors would like to thank the Israel Science Foundation for their
support of this research (grant no. 1254/18). The first-named author was partially supported by the S.A Schonbrunn Fellowship. This paper was done as part of the third-named author's master thesis under the supervision of the first- and second-named authors.}


\begin{document}

\begin{abstract}
    This paper deals with the class of existentially closed models of fields with a distinguished submodule (over a fixed subring). In the positive characteristic case, this class is elementary and was investigated by the first-named author. Here we study this class in Robinson's logic, meaning the category of existentially closed models with embeddings following Haykazyan and Kirby, and prove that in this context this class is NSOP$_1$ and TP$_2$.
\end{abstract}

\maketitle


\setcounter{tocdepth}{1}
\tableofcontents

\section{Introduction}
\iftoggle{THESIS}{
The existentially closed models of a theory are those that are existentially closed in every model extension.
Existentially closed models have a random, or generic, aspect to them by their definition --- every finite quantifier free structure that exists in some extension will also exist in the existentially closed model.
Finding first-order theories that axiomatize the class of existentially closed models is a strong tool in studying the generic models, and if the theory is inductive this will result in the model companion.

In \cite{d_Elb_e_2021_generic}, d'Elbée studied the theory of models with a generic substructure.
A particular example of interest to us is the theory of fields of positive characteristic with a distinguished vector subspace over a finite subfield, the class of existentially closed models of this theory is first-order axiomatizable, which gives rise to a model companion. 
The theory ACF$_p$G of algebraically closed fields of characteristic $p>0$ with a generic additive subgroup is a specific case of the above construction, as additive subgroups are vector subspaces over $\mathbb{F}_p$.
This theory was extensively studied in \cite{D_ELB_E_2021_acfg}.

Furthermore, \cite{d_Elb_e_2021_generic} defines \emph{weak-independence} and \emph{strong-independence}, and gives conditions for a model with a generic substructure to be NSOP$_1$, where weak-independence is Kim-independence (an introduction to those concepts can be found in the previous chapter).
The model companion of fields of positive characteristic with a vector subspace over a finite subfield satisfies those conditions, so it is NSOP$_1$.
It was also proved that this model companion is not simple (\cite[Proposition 5.20]{d_Elb_e_2021_generic}).

It is a natural to try and generalize this results to fields that are of characteristic $0$, or vector spaces that are over infinite subfields.
Another generalization is to consider modules over infinite subrings (a finite subring 
is a field).
In \cite{d_Elb_e_2021_generic}, d'Elbée showed that for fields of characteristic $0$ with an additive subgroup (which is a $\Z$-module), the class of existentially closed models is not first-order axiomatizable.

However, it is possible to study the existentially closed models of an inductive theory in a different logical setting, namely in Robinson's logic. In essence, it means that instead of studying models and elementary embeddings between them, we study existentially closed models and embeddings between them (see the introduction of \cite{Poizat2018}). Pillay \cite{Pillay2000}  refers to this setting as the \emph{Category of existentially closed models}. 
This approach was used by Haykazyan and Kirby \cite{Haykazyan_2021} in their study of exponential fields --- fields $F$ with a distinguished homomorphism $E:F^+\to F^\times$ from the additive group structure to the multiplicative group structure. 

(We note that there is a recent interest in positive model theory, which is a generalization of our setting (e.g. \cite{Haykazyan2019,hrushovski2020definability,dobrowolski2021kimindependence}).)

This chapter follows the steps of \cite{Haykazyan_2021}, considering the structure of fields with a submodule.
We will first give a description of the existentially closed fields with submodules (\cref{existentially closed}).
This description will not in general be first-order, except for the case of positive characteristic and submodules over a finite subring (see \Cref{rk_Rdefinable}).
We will then prove that the category of existentially closed models of this theory is NSOP$_1$ (see \cref{ec nsop1}) but not NTP$_2$ (and in particular, not simple; see \cref{ec tp2}); the appropriate definitions for these concepts in the category of existentially closed models appear in \cref{ec tree properties}, and are taken from \cite{Haykazyan_2021}.
The proof of NSOP$_1$ will use weak independence (mentioned above).
We will also study strong independence, which does not help with proving NSOP$_1$ but has interesting properties of its own, including $n$-amalgamation for every $n$ (see \cref{n-amalgamation}).
In the proofs we are using a definition of higher amalgamation that is slightly different from the one found in the literature (see \cite{Hrushovski98simplicityand,dePiro2006group,Goodrick2013amalgamation}). 
In the appendix we study this notion of amalgamation and its relation to independence.
}{
Suppose that $\mathcal{C}$ is a class of structures. A structure $M \in \mathcal{C}$ is existentially closed if, roughly speaking, any quantifier-free configuration that holds in an extension of $M$ from $\mathcal{C}$ is witnessed in $M$ (see \cref{existentially closed def} for the formal definition). 
Finding first-order axioms for the class of existentially closed structures in $\mathcal{C}$ could be extremely important in analyzing $\mathcal{C}$. For example, when we study the theory of fields, it is often convenient to work with algebraically closed fields, which are precisely the existentially closed fields. In an inductive theory $T$, if the class of existentially closed models is first-order axiomatizable, then the resulting theory is the model companion of $T$.

In \cite{d_Elb_e_2021_generic}, the first named author studied pairs $(M,N)$ where $N$ is a substructure of a reduct of $M$.
The existentially closed models of such pairs are called models with a generic substructure.
A particular example of interest to us is the theory of fields of positive characteristic with a distinguished vector subspace over a finite subfield. The class of existentially closed models of this theory is first-order axiomatizable, which gives rise to a model companion. 
The theory ACF$_p$G of algebraically closed fields of characteristic $p>0$ with a generic additive subgroup is a specific case of the above construction, as additive subgroups are vector subspaces over $\mathbb{F}_p$.
This theory was extensively studied in \cite{D_ELB_E_2021_acfg}.

Furthermore, \cite{d_Elb_e_2021_generic} defines \emph{weak-independence} and \emph{strong-independence}, and gives conditions implying that a model with a generic substructure is NSOP$_1$ and weak-independence is Kim-independence.
The model companion of fields of positive characteristic with a vector subspace over a finite subfield satisfies those conditions, so it is NSOP$_1$.
It was also proved that this model companion is not simple (\cite[Proposition 5.20]{d_Elb_e_2021_generic}).

It is natural to try and generalize this results to fields that are of characteristic $0$, or vector spaces that are over infinite subfields.
Another generalization is to consider modules over infinite subrings (a finite subring 
is a field).
In \cite{d_Elb_e_2021_generic}, the first named author showed that for fields of characteristic $0$ with an additive subgroup (which is a $\Z$-module), the class of existentially closed models is not first-order axiomatizable.

However, it is possible to study the existentially closed models of an inductive theory in a different logical setting, namely in Robinson's logic. In essence, it means that instead of studying models and elementary embeddings between them, we study existentially closed models and embeddings between them (see the introduction of \cite{Poizat2018}). Pillay \cite{Pillay2000}  refers to this setting as the \emph{Category of existentially closed models}. 
This approach was used by Haykazyan and Kirby \cite{Haykazyan_2021} in their study of exponential fields --- fields $F$ with a distinguished homomorphism $E:F^+\to F^\times$ from the additive group structure to the multiplicative group structure. 

(We note that there is a recent interest in positive model theory, which is a generalization of our setting (e.g. \cite{Haykazyan2019,hrushovski2020definability,dobrowolski2021kimindependence}).)

This paper follows the steps of \cite{Haykazyan_2021}, considering the structure of fields with a submodule.
We will first give a description of the existentially closed fields with submodules over a fixed ring (\cref{existentially closed}).
This description will not in general be first-order, except for the case of positive characteristic and submodules over a finite subring (see \Cref{rk_Rdefinable}). We use this description in \cref{section ultraproducts} to prove that an ultraproduct of models of ACF$_p$G over all primes $p$ belongs to the class of existentially closed fields with a vector subspace over a fixed pseudo-finite field, in spite of the fact that this class is not elementary (this is \cref{T:ultraproduct existentially closed}).

We will then prove that the category of existentially closed models of this theory is NSOP$_1$ (see \cref{ec nsop1}) but not NTP$_2$ (and in particular, not simple; see \cref{ec tp2}); the appropriate definitions for these concepts in the category of existentially closed models appear in \cref{ec tree properties}, and are taken from \cite{Haykazyan_2021}.
The proof of NSOP$_1$ will use weak independence (mentioned above).
We will also study strong independence, which does not help with proving NSOP$_1$ but has interesting properties of its own, including $n$-amalgamation for every $n$ (see \cref{n-amalgamation}). 
In the proofs we are using a definition of higher amalgamation that is slightly different from the one found in the literature (see \cite{Hrushovski98simplicityand,dePiro2006group,Goodrick2013amalgamation}). 
In the appendix we study this notion of amalgamation and its relation to independence.
\\

\noindent \textbf{Acknowledgement.} The authors would like to thank Jonathan Kirby for clarification and helpful comments, especially regarding \cref{ACF n amalgamation}, and thank Mark Kamsma for a helpful discussion. We also thank Zo\'e Chatzidakis for her comments after reading an earlier version, and especially for noticing some problems with \cref{ACF n amalgamation}.
We also thank Alex Kruckman for finding \cref{exa:Alex example} and allowing us to present it here.

We would also like to thank the anonymous referee for their careful reading and their comments.
}

\section{Preliminaries}
In this section, we will present the definitions and facts needed to work in the category of existentially closed models. 
Unless stated otherwise, all definitions and results will be given as they are presented by Haykazyan and Kirby \cite{Haykazyan_2021}.

\subsection{Existentially closed models of an inductive theory}

\begin{definition} \label{existentially closed def}
    A model $M\vDash T$ is an \emph{existentially closed} model of $T$ if for every extension $M\subseteq N\vDash T$, and every quantifier-free formula $\phi(x,a)$ with parameters $a\in M$, $N\vDash \exists x\phi(x,a)\implies M\vDash \exists x\phi(x,a)$.
\end{definition}

\begin{remarkcnt}
    If $T$ is inductive, then for every $A\vDash T$ we can construct by transfinite induction an extension $A\subseteq M$ such that $M\vDash T$ is existentially closed.
\end{remarkcnt}

Let $\mathrm{Emb}(T)$ be the category of models of $T$ with embeddings between them.
Let $\mathcal{EC}(T)$ be the full subcategory of $\mathrm{Emb}(T)$ consisting of existentially closed models and embeddings between them (which in particular preserve existential formulas). 

\begin{fact}[{\cite[Fact 2.3]{Haykazyan_2021}}] \label{companions}
    For two inductive theories $T_1$ and $T_2$, the following are equivalent
    \begin{enumerate}
        \item The theories $T_1$ and $T_2$ have the same universal consequences.
        \item Every model of $T_1$ embeds into a model of $T_2$ and vice-versa.
        \item The existentially closed models of $T_1$ and $T_2$ are the same.
    \end{enumerate}
\end{fact}
Two theories $T_1$ and $T_2$ satisfying the above equivalent conditions are called \emph{companions}.
Thus, $\mathcal{EC}(T)$ uniquely determines the theory $T$ modulo companions for $T$ inductive.

We will also use the following fact.

\begin{fact}[{\cite[Fact 2.2]{Haykazyan_2021}}] \label{maximal type}
    Let $M$ be a model of an inductive theory $T$.
    The following are equivalent.
    \begin{enumerate}
        \item $M$ is an existentially closed model of $T$.
        \item For every tuple $a\in M$, the type $\tp^M_\exists(a)$ is a maximal existential type.
    \end{enumerate}
\end{fact}

\begin{remarkcnt} \label{maximal type over set}
    In particular, if $M$ is an existentially closed model of $T$, and $A\subseteq M$ is a subset, then $\tp^M_\exists(a/A)$ is a maximal existential type over $A$.
    Indeed, let $M_A$ be the model $M$ with added constant symbols for $A$, and let $T_A$ be the same theory as $T$ but in the expanded language.
    Every model of $T_A$ extending $M_A$ must interpret the constant symbols as $A$, so $M_A$ is an existentially closed model of $T_A$, as we allow parameters in the definition of existentially closed.
    The result then follows from \cref{maximal type}.
\end{remarkcnt}

\subsection{Amalgamation and joint embedding}

\begin{definition}
    A model $A\vDash T$ is an \emph{amalgamation base} for $\mathrm{Emb}(T)$ if for every two models $B_1,B_2\vDash T$ and embeddings $f_1:A\to B_1$ and $f_2:A\to B_2$, then there is a model $C\vDash T$ and embeddings $g_1:B_1\to C$ and $g_2:B_2\to C$ such that $g_1\circ f_1=g_2\circ f_2$.
    
    Furthermore, $A$ is a \emph{disjoint amalgamation base} if we can pick $g_1,g_2$ such that $g_1(B_1)\cap g_2(B_2)=g_1(f_1(A))$. 
\end{definition}

\begin{fact}[{\cite[Corollary 8.6.8]{hodges1993model}}]
Every existentially closed model is a disjoint amalgamation base.
\end{fact}

\begin{definition}
    The category $\mathrm{Emb}(T)$ has the \emph{joint embedding property} (JEP) if any two models of $T$ can be embedded into a third model.
\end{definition}

In the category of existentially closed models, extending $T$ to an inductive theory $T'$ with JEP corresponds to choosing a completion in first-order logic.
However, we need to make sure that $\mathcal{EC}(T')$ is contained in $\mathcal{EC}(T)$.
This gives rise to the following definition.

\begin{definition}
    An inductive extension $T'$ of an inductive theory $T$ is called a \emph{JEP-refinement} of $T$ if $T'$ has JEP and every existentially closed model of $T'$ is an existentially closed model of $T$.
\end{definition}

\begin{fact}[{\cite[Lemma 2.12]{Haykazyan_2021}}] \label{amalgamation base refinement}
    If $A$ is an amalgamation base for $\mathrm{Emb}(T)$, then $T\cup \mathrm{Th}_\exists(A)$ is a JEP-refinement of $T$.
    
    Furthermore, an existentially closed model of $T$ is a model of a unique JEP-refinement of $T$ modulo companions.
\end{fact}

\subsection{Higher amalgamation}\label{higher amalgamtion}

We proceed to define higher amalgamation, as it was defined in \cite{Haykazyan_2021}.\\

Let $\mathcal{K}\subseteq \mathrm{Emb}(T)$ be a subcategory.
Let $n\ge 3$, consider $n$ as a set $n=\set{0,\dots,n-1}$ and consider $\mathcal{P}(n)$ and $\mathcal{P}^{-}(n)=\mathcal{P}(n)\setminus\set{n}$ as categories with a unique morphism $a\to b$ if $a\subseteq b$.
Define a $\mathcal{P}(n)$-system (respectively, $\mathcal{P}^-(n)$-system)  of $\mathcal{K}$ to be a functor $F$ from $\mathcal{P}(n)$ (respectively, $\mathcal{P}^{-}(n)$) to $\mathcal{K}$. 
For each $a\in \mathcal{P}(n)$, denote $F_a=F(a)$.

Suppose that for every $M\in \mathcal{K}$, we have a ternary relation $\forkindep$ on subsets of $M$. 
A $\mathcal{P}(n)$-system ($\mathcal{P}^-(n)$-system) $F$ is called \emph{independent} with respect to $\forkindep$, if for every $a\in \mathcal{P}(n)$ ($a\in \mathcal{P}^-(n)$) and every $b\subseteq a$,
$$F_b\forkindep_{\bigcup_{c\subsetneq b} F_c}\bigcup_{b\not\subseteq d\subseteq a} F_d$$
as subsets of $F_a$, where we consider every embedding $F_b\to F_a$ as an inclusion.

\begin{definition}
    Suppose $\mathcal{K},\forkindep$ are as above.
    Say that $\mathcal{K}$ has \emph{$n$-amalgamation} ($n\ge 3$) if any independent $\mathcal{P}^-(n)$-system in $\mathcal{K}$ can be completed to an independent $\mathcal{P}(n)$-system.
    Say that $T$ has $n$-amalgamation if $\mathrm{Emb}(T)$ has $n$-amalgamation.
\end{definition}

Note that this definition of independent systems and $n$-amalgamation is not the same as the one used by other authors, e.g. \cite{Hrushovski98simplicityand,dePiro2006group,Goodrick2013amalgamation}.
It is, however, similar to the definition of stable systems found in \cite{shelah1990classification}, with the main difference being that in stable systems all embeddings are inclusions and everything lives inside the monster model, so there is no amalgamation.
This enables us to use the following fact, which is originally stated for general stable theories, but will be presented here as in \cite[Fact 5.3]{Haykazyan_2021} where it is stated specifically for ACF. 

\begin{fact}[{\cite[Fact XII.2.5]{shelah1990classification}}]\label{fact shelah}
    Let $F=\set{F_s}_{s\subseteq n}$ be an independent $\mathcal{P}(n)$-system of ACF, where every $F_s$ is considered as a subset of $F_n$, and let $t\subseteq n$.
    For $i<m$ let $s(i)\in \mathcal{P}(n)$ and let $\bar{a}_i\in F_{s(i)}$.
    Assume that for some formula $\phi(\bar{x}_0,\dots,\bar{x}_{m-1})$ we have $F_n\vDash \phi(\bar{a}_0,\dots,\bar{a}_{m-1})$.
    Then there are $\bar{a}_i'\in F_{s(i)\cap t}$ such that $F_n\vDash \phi(\bar{a}_0',\dots,\bar{a}_{m-1}')$, and if $s(i)\subseteq t$, then $\bar{a}_i'=\bar{a}_i$.
\end{fact}
In \cref{Results on higher amalgamation} we prove some well known results about higher amalgamation using our definition, including the fact that ACF has $n$-amalgamation for every $n$ (\cref{ACF n amalgamation}).

\subsection{Monster model}
We present a notion of saturation for the category of existentially closed models.
It is convenient to work inside a large saturated model, which we will call a monster model.

This section borrows from \cite[\S 2.4]{Haykazyan_2021}, except for our definition of strong $\kappa$-homogeneity and \cref{saturation implies}, see \cref{HK strong homogeneity}.

\begin{definition}
    Let $T$ be an inductive theory with JEP, and suppose $M$ is an existentially closed model of $T$. 
    Let $\kappa$ be a cardinal.
    \begin{itemize}
        \item $M$ is called \emph{$\kappa$-saturated} if every unitary existential type with parameters from a set $A\subseteq M$ of cardinality less than $\kappa$ is realized in $M$.
        \item $M$ is called
        \emph{$\kappa$-universal} if for every $A\vDash T$, and a tuple $a\subseteq A$ with $\abs{a}<\kappa$, there exists a tuple $b\subseteq M$ realizing $\tp_\exists^A(a)$ 
        (that is, $\tp_\exists^A(a)\subseteq \tp_\exists^M(b)$).
        \item $M$ is called \emph{$\kappa$-homogeneous} if for any two tuples $a,b$  from $M$ with length less than $\kappa$ such that $a\equiv^\exists b$, and every singleton $a'\in M$, there exists a singelton $b'\in M$ such that $aa'\equiv^\exists bb'$.
        \item $M$ is called \emph{strongly $\kappa$-homogeneous}  if for any two tuples $a,b$  from $M$ with length less than $\kappa$ such that $a\equiv^\exists b$, there exists an automorphism $\sigma$ of $M$ such that $\sigma(a)=b$.   
    \end{itemize}
\end{definition}

\begin{remarkcnt}
    If $\kappa>\abs{L}$, then $M$ is $\kappa$-universal iff every model $A\vDash T$ of size less than $\kappa$ can be embedded in $M$, by Löwenheim-Skolem.
\end{remarkcnt}

\begin{proposition} \label{saturation implies}
    In the same settings as above, the following are equivalent:
    \begin{enumerate}
        \item $M$ is $\kappa$-saturated,
        \item $M$ is $\kappa^+$-universal and $\kappa$-homogeneous 
        \item $M$ is $\aleph_0$-universal and $\kappa$-homogeneous
    \end{enumerate}
    Furthermore, if $\kappa=\abs{M}$, then $\kappa$-homogeneity implies strong $\kappa$-homogeneity.
\end{proposition}

\begin{proof}
    $(1)\implies (2)$:
    Suppose $M$ is $\kappa$-saturated.
    To prove $\kappa^+$-universality, let $A\vDash T$ and let $a=(a_i)_{i< \kappa}\subseteq A$ be a tuple.
    For $\alpha\le \kappa$, denote $a_{<\alpha}=(a_i)_{i<\alpha}$.
    We will construct $b=(b_i)_{i< \kappa}$ satisfying $\tp^A_\exists(a)$, by constructing $b_{<\alpha}$ by induction on $\alpha\le \kappa$.
    For $\alpha=1$, by JEP there is some model $N\vDash T$ and embeddings $f_1:A\to N$ and $f_2:M\to N$.
    $f_1(a_0)$ realizes $\tp^A_\exists(a_0)$ in $N$, as it is an existential type.
    Because $M$ is existentially closed and embeds in $N$, it follows that $\tp^A_\exists(a_0)$ is consistent in $M$, and there is $b_0\in M$ realizing it by saturation.
    For $\alpha+1$, consider the existential type $p(x)=\tp^A_\exists(a_\alpha/a_{<\alpha})$, replacing the parameters $a_{<\alpha}$ with $b_{<\alpha}$ results in a consistent existential type $q(x)$ in $M$, because for every finite conjunction $\psi(x,b_{<\alpha})$ of formulas from $q(x)$, we have $A\vDash \exists x \psi(x,a_{<\alpha})$, so $M\vDash \exists x \psi(x,b_{<\alpha})$.
    By saturation there is some $b_\alpha\in M$ satisfying $q(x)$, so $B_{<\alpha}b_\alpha=b_{<\alpha +1}$ satisfies $\tp^A_\exists(a_{<\alpha+1})$.
    If $\alpha$ is a limit ordinal, take the union $b_{<\alpha}=\bigcup_{\beta<\alpha}b_{<\beta}$. 
    To prove $\kappa$-homogeneity, suppose $a,b\subseteq M$ are tuples of length less than $\kappa$, and let $a'\in M$.
    Consider $p(x)=\tp^M_\exists(a'/a)$, replacing the parameters $a$ by $b$ results in a consistent existential type, because for every finite conjunction $\psi(x,b)$ of formulas from $q(x)$, we have $A\vDash \exists x \psi(x,a)$, so $M\vDash \exists x \psi(x,b)$.   
    
    $(2)\implies (3)$:
    trivial.
    
    $(3)\implies(1)$:
    First we will prove that $\kappa$-universality and $\kappa$-homogeneity imply $\kappa$-saturation:
    Let $p(x)$ be a unitary existential type ($\abs{x}=1$) over $A\subseteq M$ of size less than $\kappa$.
    There is some extension $N\supseteq M$ with an element $b\in N$ realizing $p(x)$.
    By $\kappa$-universality, there is some $A'b'\subseteq M$ that satisfy $\tp_\exists^N(Ab)$, when considered as tuples.
    In particular, we have $A'\equiv^\exists A$ in $M$, so by $\kappa$-homogeneity there is some $b''\in M$ such that $A'b'\equiv^\exists Ab''$.
    Thus, $b''\vDash \tp_\exists^M(b/A)=p(x)$.
    
    Now, assuming $\aleph_0$-universality and $\kappa$-homogeneity, will prove by induction on $\lambda\le \kappa$ that $M$ is $\lambda$-saturated.
    For $\lambda=\aleph_0$, it follows from the above claim.
    For $\lambda^+$, we know that $M$ is $\lambda$-saturated, so by $(1)\implies (2)$ $M$ is $\lambda^+$-universal. 
    We also know that $M$ is $\lambda^+$-homogeneous, so by the claim $M$ is $\lambda$-saturated.
    For $\lambda$ a limit cardinal, a set of parameters $A\subseteq M$ of size less than $\lambda$, is also of size less than $\mu$ for some $\mu<\lambda$.
    
    For the ``furthermore'' part, if $M$ is $\abs{M}$-homogeneous and $a\equiv^\exists b$ in $M$, we can construct an automorphism mapping $a$ to $b$ by the back and forth method.
\end{proof}

\begin{remarkcnt} \label{HK strong homogeneity}
    Our definition of strong $\kappa$-homogeneity differs from the one given in \cite{Haykazyan_2021}, which is
    \begin{itemize}
        \item $M$ is called \emph{strongly $\kappa$-homogeneous} if for any amalgamation base $A$ of size less than $\kappa$ and embeddings $f_1,f_2$ of $A$ in $M$, there exists an automorphism $\sigma$ of $M$ such that $\sigma\circ f_1=f_2$.   
    \end{itemize}
    However, strong $\kappa$-homogeneity in our definition implies strong $\kappa$-homogeneity in their definition:
    $A$ is an amalgamation base, so there is a model $N\vDash T$ and embeddings $g_1, g_2$ of $M$ in $N$ such that $g_1\circ f_1=g_2\circ f_2$.
    With $A$ considered as a tuple, we have 
    $$\tp^M_\exists(f_1(A))=\tp^N_\exists(g_1(f_1(A)))=\tp^N_\exists(g_2(f_2(A)))=\tp^M_\exists(f_1(A)),$$
    because $M$ is existentially closed.
    From our definition of strong homogeneity, there is an automorphism $\sigma$ of $M$ such that $\sigma(f_1(A))=f_2(A)$ considered as tuples, thus $\sigma\circ f_1=f_2$.
\end{remarkcnt}

The following example (due to Alex Kruckman) shows that the opposite direction of \cref{HK strong homogeneity} does not hold, namely, strong $\kappa$-homogeneity in the sense of \cite{Haykazyan_2021} does not even imply $\aleph_0$-homogeneity in our sense.

\begin{example} \label{exa:Alex example}
    Let $V$ be a vector space over $\F_2$ of countable dimension, and let $(e_i)_{i<\omega}$ be a basis. 
    Consider the two-sorted language $L$ with sorts $X$ and $Y$ containing unary function symbols $f^a :X \to X$ for every $a \in V$, unary predicate symbols $(P^i)_{i<\omega}$ on $X$ and a function symbol $\pi:X\to Y$.
    
    Let $T$ be the $L$-theory whose axioms are:
    \begin{enumerate}
        \item $Y$ is infinite, and for each $y\in Y$, $\pi^{-1}(y)$ is infinite.
        \item $V$ acts freely on $X$ (via the functions $f_a$).
        \item For each $x\in X$ and $i\ne j<\omega$, $P^i(x)\iff P^i(f^{e_j}(x))$ and $P^i(x)\iff \lnot P^i(f^{e_i}(x))$.
        \item For each $x\in X$ and $a\in V$, $\pi(x)=\pi(f^a(x))$.
    \end{enumerate}
    It is not difficult to see that $T$ is complete, has quantifier elimination and is thus inductive (in essence, each fiber of $\pi$ is a model of the theory described in \cite{kruckman2017}). Also, as $T$ is model complete, every model of $T$ is existentially closed.
    
    Let $\sim$ be the equivalence relation on $2^{\omega}$ (functions from $\omega$ to $\{0,1\}$) of agreeing up to finitely many coordinates ($\eta \sim \nu$ iff for some $n<\omega$, $\eta(i)=\nu(i)$ for all $i>n$). Note that there are continuum many $\sim$-classes. 
    
    Construct the following model $M$ of $T$: $X_M = 2^\omega$, $Y_M = 2^\omega/\mathord{\sim}$, $\pi_M:X_M\to Y_M$ is the canonical projection, $f^{e_i}_M(\eta)$ flips the $i$'th coordinate of $\eta\in X_M$ and $P^i_M(\eta)$ holds iff $\eta(i)=1$.
    Note that no two elements of $X_M$ have the same type. This implies (*): if $N \models T$ embeds into $M$, there is a unique embedding witnessing this (and in particular $M$ is rigid). On the one hand, (*) implies that $M$ is strongly $2^{\aleph_0}$-homogeneous in the sense of \cite{Haykazyan_2021}, and on the other hand, by quantifier elimination all elements of $Y_M$ have the same type, so $M$ is not $\aleph_0$-homogeneous (because if $w\neq r \in Y_M$, given some $\eta \in w$ there is no $\nu\in r$ such that $\eta\equiv \nu$).
\end{example}





Call $M$ saturated if it is $\abs{M}$-saturated.
We will call a large saturated model a \emph{monster model}.
In these settings, monster models are often called universal domains, or e-universal domains, but we kept the terminology of \cite{Haykazyan_2021}.

We will assume that monster models exist. This usually requires some set theoretic assumptions like the generalized continuum hypothesis, but one can change the set-theoretic universe without changing any object we are interested in, ensuring that monster models of large enough sizes exist (see \cite{halevi2021saturated}). One can also work without a monster model, using only commuting diagrams, but it is less convenient.

\subsection{Model theoretic tree properties} \label{ec tree properties}
We will present two properties of formulas, TP$_2$ and SOP$_1$, adapted to the category of existentially closed models.
The main difference is that the formulas have to be existential, and there must be an existential formula that witnesses inconsistency. 
In the following, let $T$ be an inductive theory with JEP, and work inside a monster model $\M\vDash T$.

\begin{definition} \label{definition tp2 exist}
  An existential formula $\phi(x,y)$ ($x,y$ tuples) has TP$_2$ with respect to $\mathcal{EC}(T)$ if there is an amalgamation base $A\vDash T$, an existential formula $\psi(y_1,y_2)$ and parameters $(a_{i,j})_{i,j<\omega}$ from $A$ such that the following hold:
  \begin{enumerate}
      \item for all $\sigma\in \omega^\omega$, the set $\set{\phi(x,a_{i,\sigma(i)})}$ is consistent.
      \item $\psi(y_1,y_2)$ implies that $\phi(x,y_1)\land \phi(x,y_2)$ is inconsistent, i.e.
      $$T\vdash \lnot\exists xy_1y_2 [\psi(y_1,y_2)\land \phi(x,y_1)\land \phi(x,y_2)]$$
      \item for every $i,j,k<\omega$, if $j\ne k$, then $A\vDash \psi(a_{i,j},a_{i,k})$.
  \end{enumerate}
  
  If no existential formula has TP$_2$, we say that $\mathcal{EC}(T)$ is NTP$_2$.
\end{definition}

\begin{definition} \label{definition sop1 exist}
  An existential formula $\phi(x,y)$ ($x,y$ tuples) has SOP$_1$ with respect to $\mathcal{EC}(T)$ if there is an amalgamation base $A\vDash T$, an existential formula $\psi(y_1,y_2)$ and a binary tree of parameters $(a_\eta)_{\eta\in 2^{<\omega}}$ from $A$ such that the following hold:
  \begin{enumerate}
      \item for every branch $\sigma\in 2^{\omega}$, the set $\set{\phi(x,a_{\sigma|_n})}$ is consistent.
      \item $\psi(y_1,y_2)$ implies that $\phi(x,y_1)\land \phi(x,y_2)$ is inconsistent, i.e.
      $$T\vdash \lnot\exists xy_1y_2 [\psi(y_1,y_2)\land \phi(x,y_1)\land \phi(x,y_2)]$$
      \item for all $\eta\in 2^{<\omega}$, if $\nu\unrhd\eta\frown \vect{0}$, then $A\vDash \psi(a_{\eta\frown \vect{1}}, a_\nu)$.
  \end{enumerate}
  
  If no existential formula has SOP$_1$, we say that $\mathcal{EC}(T)$ is NSOP$_1$.
\end{definition}

\begin{remarkcnt}
    If the class $\EC(T)$ is first-order axiomatizable by $T'$, that is $T'$ is the model companion of $T$, then the above definitions are equivalent to $T'$ being NTP$_2$, NSOP$_1$ respectively in the usual first-order sense. To be precise, an existential formula $\phi(x,y)$ witnesses SOP$_1$/TP$_2$ (as defined in \cite[Definition 3.1]{Che14}) in $T'$ iff $\phi(x,y)$ witnesses SOP$_1$/TP$_2$ in $\EC(T)$. This is essentially because in model complete theories, every formula is equivalent to an existential one.  
\end{remarkcnt}

There is a Kim-Pillay style characterization for NSOP$_1$ theories in the category of existentially closed models. 
This characterization is due to Haykazyan and Kirby \cite{Haykazyan_2021}, and is based on a theorem of Chernikov and Ramsey \cite{Chernikov2016} for complete first-order theories.

\begin{fact}[{\cite[Theorem 6.4]{Haykazyan_2021}}] \label{nsop1 condition exist}
    Let $\forkindep$ be a $\mathrm{Aut}(\M)$-invariant ternary relation on small subsets of $\M$.
    Assume that $\forkindep$ satisfies the following, for any small existentially closed model $M$ and tuples $a,b$ from $\M$:
    \begin{itemize}
        \item (Strong finite character) if $a\not\forkindep_M b$, then there is an existential formula $\phi(x,b,m)\in \tp_\exists(a/Mb)$ such that for any $a'$ realizing $\phi$, the relation $a'\not\forkindep_M b$ holds.
        \item (Existence over models) $a\forkindep_M M$
        \item (Monotonicity) $aa'\forkindep_M bb'$ implies $a\forkindep_M b$.
        \item (Symmetry) $a\forkindep_M b$ implies $b\forkindep_M a$.
        \item (Independence theorem) If $c_1\forkindep_M c_2$, $b_1\forkindep_M c_1$, $b_2\forkindep_M c_2$ and $b_1\equiv^\exists_M b_2$ then there exists $b$ with $b\equiv^\exists_{Mc_1} b_1$ and $b\equiv^\exists_{Mc_2} b_2$.
    \end{itemize}
    Then $\mathcal{EC}(T)$ has NSOP$_1$.
\end{fact}

It is folklore that the independence theorem is equivalent to 3-amalgamation, see e.g. the discussion under \cite[Definition 9.1.3]{kim2014simplicity}.
However, different definitions of amalgamation are used by different authors, as noted in the beginning of \cref{higher amalgamtion}. 
For this reason we include in the Appendix a proof of this equivalence in our setting (\Cref{equivalence 3 amalgamation}).

\section{Special models of fields with a submodule} 
In this section, we will define the theory of fields with a submodule, and give a characterization of special models that we are interested in: existentially closed models and amalgamation bases.

\subsection{Existentially closed models}
For the rest of the paper, let $R$ be an integral domain.

\begin{lemma} \label{R-modules intersection}
    If $A,B,C$ are $R$-modules, that are submodules of a common $R$-module, such that $B\subseteq A$, then $A\cap(B+C)=B+(A\cap C)$.
\end{lemma}

\begin{proof}
    It is clear that $B+(A\cap C)\subseteq A\cap(B+C)$. 
    For the other inclusion, suppose $a\in A\cap (B+C)$.
    There are $b\in B$ and $c\in C$ such that $a=b+c$, but $b\in B\subseteq A$, so we get $c=a-b\in A\cap C$. 
    Thus, $a\in B+(A\cap C)$.
\end{proof}

\begin{definition}
    Let $L_{R;P}$ be the language of rings with a constant symbol for every element $r\in R$, and a unitary predicate $P$.
    \begin{itemize}
        \item Define the theory F$_{R\text{-module}}$\footnote{F stands for the theory of fields, as ACF stands for the theory of algebraically closed fields.} in the language $L_{R;P}$ to be the theory of fields with the quantifier-free diagram of $R$, and $P$ an $R$-submodule of the universe.
        \item Define the theory ACF$_{R\text{-module}}$ in the language $L_{R;P}$ to be the extension of F$_{R\text{-module}}$ with the added axioms that the universe is an algebraically closed field ($\ACF_{R\text{-module}}=\ACF\cup \mathrm{F}_{R\text{-module}}$).
    \end{itemize}
\end{definition}

\begin{remarkcnt}
    If $M,N\vDash \mathrm{F}_{R\text{-module}}$, then
    \begin{enumerate}
        \item $R$ is a subring of $M$.
        \item $M$ is an $L_{R;P}$-substructure of $N$ iff $M$ is a subfield of $N$ and $P_N\cap M=P_M$ 
    \end{enumerate}
\end{remarkcnt}

\begin{example}
    If $R=\Z$, then F$_{\Z-\text{module}}$ is the theory of fields of characteristic $0$ with an additive subgroup.
    If $R=\Q$, then F$_{\Q-\text{module}}$ is the theory of fields of characteristic $0$ with a divisible additive subgroup.
    If $R=\F_p$, then F$_{\F_p-\text{module}}$ is the theory of fields of characteristic $p$ with an additive subgroup, which was studied in \cite{d_Elb_e_2021_generic}.
\end{example}

\begin{definition}
   Let $K$ be a field containing $R$. 
   Call a variety $V\subseteq K^n$ \emph{$R$-free} if there is some field extension $K'\supseteq K$, and $a\in K'^n$ a generic point of $V$ over $K$, such that $a$ is $R$-linearly independent over $K$.
   That is, if $r_0a_0+\dots+r_{n-1}a_{n-1}\in K$ for $r_i\in R$, then $r_0=\dots=r_{n-1}=0$.
\end{definition}

\begin{definition}
   Let $M\vDash \mathrm{F}_{R\text{-module}}$ be a field with an $R$-submodule, and let $0\le k\le n$, $0\le s$. For a matrix $A\in \Mat_{n\times s}(R)$ and a tuple $c\in M^n$, call $(A,c)$ a \emph{$k$-compatible pair} if for every $r_0,\dots,r_{k-1}\in R$,
    \begin{enumerate}
        \item $r_0A_0+\dots+r_{k-1}A_{k-1}=0\implies r_0c_0+\dots+r_{k-1}c_{k-1}\in P_M$,
        \item for $k\le i< n$, either $A_i\ne r_0A_0+\dots+r_{k-1}A_{k-1}$, or  $r_0c_0+\dots+r_{k-1}c_{k-1}-c_i\notin P_M$,
    \end{enumerate}
    where $A_i$ is the $i$-th row of $A$.
\end{definition}

\begin{theorem} \label{existentially closed}
    Let $M\vDash \mathrm{F}_{R\text{-module}}$ be a field with an $R$-submodule. 
    The model $M$ is existentially closed iff for every $R$-free variety $V\subseteq M^s$, $0\le k\le n$ and $(A,c)$ a $k$-compatible pair, where $A\in \Mat_{n\times s}(R)$ and $c\in M^n$, there is a point $b\in V$ such that for $a=Ab+c$, $a_0,..,a_{k-1}\in P_M$ and $a_k,..,a_{n-1}\notin P_M$. 
\end{theorem}

\begin{proof}
    For the left to right implication, suppose $V,A,c$ are as above.  
    There is some field extension $M'\supseteq M$ and $b'\in M'$ a generic point of $V$ over $M$ such that $b'$ is $R$-linearly independent over $M$. 
    Let $a'=Ab'+c$, and consider $M'$ as a model of F$_{R\text{-module}}$, with $P_{M'}=P_M+\vect{a'_0,\dots,a'_{k-1}}_R$. 
    To show that $M$ is an $L_{R;P}$-substructure of $M'$, we need to show that $P_{M'}\cap M=P_M$.
    By \Cref{R-modules intersection}, $P_{M'}\cap M=(P_M+\vect{a'_0,\dots,a'_{k-1}}_R)\cap M=P_M+(\vect{a'_0,\dots,a'_{k-1}}_R\cap M)$, so it is enough to show $\vect{a'_0,\dots,a'_{k-1}}_R\cap M\subseteq P_M$. 
    Suppose $m\in \vect{a'_0,\dots,a'_{k-1}}_R\cap M$, we can write $m=r_0a'_0+\dots+r_{k-1}a'_{k-1}$ with $r_i\in R$. 
    Substitute $a'_i$ for $A_ib'+c_i$ and rearrange to get
    $$(r_0A_0+\dots+r_{k-1}A_{k-1})b'=m-(r_0c_0+\dots+r_{k-1}c_{k-1})\in M.$$
    However, $b'$ is $R$-linearly independent over $M$, so we must have $r_0A_0+\dots+r_{k-1}A_{k-1}=0$. 
    This implies that $m=r_0c_0+\dots+r_{k-1}c_{k-1}$, and by $k$-compatibility $r_0c_0+\dots+r_{k-1}c_{k-1}\in P_M$, so $m\in P_M$. 
   
    Consider the formula
    $$\phi(y)=V(y)\land \bigwedge_{i< k}A_iy+c_i\in P\land \bigwedge_{k\le i<n} A_iy+c_i\notin P$$
    where $\abs{y}=\abs{b'}$.
    We want to show that $\phi(y)$ is satisfied by $b'$ in $M'$.
    It is obvious that $b'\in V(M')$ and $A_ib'+c_i=a'_i\in P_{M'}$ for $i< k$,
    it remains to prove that $A_ib'+c_i=a'_i\notin P_{M'}$ for $k\le i< n$.
    Assume towards contradiction that $a_i'\in P_{M'}$ for some $k\le i< n$, then there are $r_0,\dots,r_{k-1}\in R$ and $p\in P_M$ such that $a'_i=p+r_0a'_0+\dots+r_{k-1}a'_{k-1}$. 
    Substitute $a'_j$ for $A_jb'+c_j$ and rearrange to get
    $$(A_i-r_0A_0+\dots-r_{k-1}A_{k-1})b'= p +r_0c_0+\dots+r_{k-1}c_{k-1}-c_i\in M.$$
    Again, because $b'$ is $R$-linearly independent over $M$, this implies $A_i-r_0A_0+\dots-r_{k-1}A_{k-1}=0$, so $A_i=r_0A_0+\dots+r_{k-1}A_{k-1}$. 
    It also follows that $p+r_0c_0+\dots+r_{k-1}c_{k-1}-c_i=0$, so $ r_0c_0+\dots+r_{k-1}c_{k-1}-c_i=-p\in P_M$, in contradiction to $k$-compatibility. 
    $\phi(y)$ is satisfied by $b'$ in $M'$, so by existential closedness there is some $b\in V(M)$, such that $a=Ab+c$ satisfies $a_0,..,a_{k-1}\in P_M$, $a_k,..,a_{n-1}\notin P_M$, as needed.
    
    For the right to left implication, let $M\vDash \mathrm{F}_{R\text{-module}}$ satisfy the right-hand condition, and let $M'\vDash \mathrm{F}_{R\text{-module}}$ be some model extending $M$.
    We need to show that for every formula $\psi(x)$ which is a conjunction of atomic formulas, where $x=(x_0,\dots,x_{n-1})$ is a tuple of variables, if $M'\vDash \exists x \psi(x)$, then $M\vDash \exists x \psi(x)$.
    Atomic formulas in $L_{R;P}$ take one of the following forms:
    \begin{enumerate}
        \item $q(x)=0$,
        \item $q(x)\ne0$,
        \item $q(x)\in P$,
        \item $q(x)\notin P$,
    \end{enumerate}
    where $q(x)$ is a polynomial over $M$.
    Let $\psi(x)$ be a conjunction of atomic formulas.
    By introducing more variables, we can replace the atomic formulas of the second form $q(x)\ne 0$ with $x_n\cdot q(x)=1$, to get an atomic formula of the first form, because $\exists x (q(x)\ne 0)\iff \exists x,x_n (x_n\cdot q(x)=1)$.
    Similarly we can replace the third and fourth forms $q(x)\in P$, $q(x)\notin P$ with $q(x)=x_n\land x_n\in P$, $q(x)=x_n\land x_n\notin P$.
    After those replacements, we are left only with atomic formulas of the forms $q(x)=0$, $x_i\in P$, and $x_i\notin P$. 
    
    Furthermore, suppose $a'\in M'$ witnesses the existence $M'\vDash \exists x \psi(x)$.
    For every $i<n$, either $a'_i\in P$ or $a'_i\notin P$.
    Taking the conjunction of $\psi$ with the corresponding conditions $x_i\in P$ or $x_i\notin P$, we get a stronger formula $\psi^*(x)$ that is still satisfied by $a'$, and has the additional property that for every $i<n$ either $x_i\in P$ or $x_i\notin P$ appears in $\psi^*(x)$.
    Thus, it is enough to prove existential closedness for formulas with the above property, and we can assume that $\psi(x)$ is of the form
    $$\psi(x)=W(x)\land x_0,..,x_{k-1}\in P\land x_k,..,x_{n-1}\notin P$$
    where $W(x)$ is a conjunction of polynomial equations, i.e. a variety over $M$.
    
    Let $a'\in M'$ witness the existence $M'\vDash \exists x \psi(x)$.
    Consider the fraction field of $R$, $\fr(R)\subseteq M$.
    There is some $0\le s\le n$, and some tuple $b'\in M'^s$ that is $\fr(R)$-linearly independent over $M$, such that $M+\vect{a'}_{\fr(R)}=M\oplus\vect{b'}_{\fr(R)}$. 
    Write $a'=Ab'+c$ for $A\in \Mat_{n\times s}(\fr(R))$, $c\in M^n$.
    We can assume without loss of generality that $A$ is a matrix over $R$, else let $0\ne d\in R$ be the product of the denominators of all elements in $A$, and replace $A,b'$ with $dA,\frac{1}{d}b'$ to get a matrix over $R$.
    
    We will prove that $(A,c)$ is a $k$-compatible pair. Let $r_1,\dots,r_k\in R$, if $r_1A_1+\dots+r_kA_k=0$, then 
    \begin{align*}
        P_{M'}\ni r_1a'_1+\dots+r_ka'_k&=(r_1A_1+\dots+r_kA_k)b'+r_1c_1+\dots+r_kc_k\\ 
        &=r_1c_1+\dots+r_kc_k
    \end{align*}
    but also $r_1c_1+\dots+r_kc_k\in M$, so $r_1c_1+\dots+r_kc_k\in P_M$. 
    Suppose for $k\le i<n$ that both $A_i=r_1A_1+\dots+r_kA_k$ and $r_1c_1+\dots+r_kc_k-c_i\in P_M$, then 
    \begin{align*}
        r_1a'_1+\dots+r_ka'_k-a'_i&=(r_1A_1+\dots+r_kA_k-A_i)b'+r_1c_1+\dots+r_kc_k-c_i\\ 
        &=r_1c_1+\dots+r_kc_k-c_i\in P_M.
    \end{align*}
    Thus, $a'_i\in P_M+\vect{a'_1,..,a'_k}_R\subseteq P_{M'}$, a contradiction.
    Let $V$ be the locus of $b'$ over $M$.
    The pair $(A,c)$ is $k$-compatible and $V$ is $R$-free, so by our assumption there exists $b\in V$, such that for $a=Ab+c$ we have $a_1,..,a_k\in P_M$, $a_{k+1},..,a_n\notin P_M$. 
    Furthermore, $W(Ay+c)$ is contained in $V(y)$, as $a'=Ab'+c$ belongs to $W$, so $a=Ab+c\in W$.
    We found $a\in M^n$ such that $M\vDash \psi(a)$, as needed.
\end{proof}

\begin{remarkcnt}\label{enough to consider irreducible}
    Note that the statement of \cref{existentially closed} is equivalent to the same statement quantifying over irreducible $R$-free varieties; indeed, a variety $V$ is $R$-free iff one of its irreducible component is $R$-free, because any generic point of $V$ is the generic point of one of its irreducible components.
\end{remarkcnt}

\begin{remarkcnt}\label{rk_Rdefinable}
    Let $M$ be an existentially closed model of F$_{R-\text{module}}$. Then $R$ is definable as a subset of $M$. Indeed, $R = \set{x\in M \mid xP_M\subseteq P_M}$ (i.e. defined by the universal formula $\forall y\in P\ xy\in P$): `$\subseteq $' is clear. For the other direction assume by contradiction that $m\in M\setminus R$ and $mP_M\subseteq P_M$, and consider the structure $N = (\overline{M(t)},P_M+R\frac{t}{m})$, where $t$ is transcendental over $M$. Then it is easy to check that $N$ is an extension of $M$ and that $t\notin P_N$. Thus, $N\vDash \exists x (x\in P \wedge mx\notin P)$ (the second conjunct uses the assumption towards contradiction), and since $M$ is existentially closed, we get a contradiction. This is essentially the same proof as \cite[Proposition 5.32]{d_Elb_e_2021_generic}.
    
    It follows that if $R$ is infinite, then the class of existentially closed models of F$_{R-\text{module}}$ is not elementary --- else, starting with some existentially closed model, we could construct by compactness an existentially closed model extending it with a strictly larger $R$. On the other hand, if $R$ is finite, then this class is elementary, by \cite[Proposition 5.4]{d_Elb_e_2021_generic}.
    In particular, if $R$ is infinite, the characterization of existentially closed models given in \Cref{existentially closed} is not first-order.
\end{remarkcnt}

\subsection{Ultraproducts of existentially closed models}\label{section ultraproducts}

In \cite{D_ELB_E_2021_acfg}, the theory of fields of characteristic $p>0$ with a predicate for an additive subgroup, which is exactly $\mathrm{F}_{\F_p-module}$, was shown to have a model companion named ACF$_p$G.
It was also shown in \cite[Remark 1.20]{D_ELB_E_2021_acfg} that an ultraproduct of models of ACF$_p$G, with $\mathcal{U}$ a non-principal ultrafilter on the primes, is \emph{not} existentially closed as a field with an additive subgroup.
Here we will show that the ultraproduct \emph{is} existentially closed as a model of $\mathrm{F}_{K-module}$ (fields with a $K$-vector subspace), where $K=\prod_{p\in\mathrm{Primes}} \F_p/\mathcal{U}$ is a pseudo-finite field of characteristic 0.

\subsubsection*{Setting.}
Let $I$ be an indexing set, and for every $\alpha\in I$ let $R_\alpha$ be a ring and $M_\alpha\vDash\mathrm{F}_{R_\alpha-module}$.

Let $L_{Q,P}$ be the language of rings with two added predicates $Q$ and $P$, and let $\mathcal{M}_\alpha$ be the reduct to $L_{Q,P}$ of the expansion of $M_\alpha$ given by interpreting $Q$ as $R_\alpha$ (i.e. replacing the constants for $R_\alpha$ by a predicate for it).
Let $\mathcal{U}$ be a non-principal ultrafilter on $I$ and define $\mathcal{M}=\prod_{\alpha\in I} \mathcal{M}_\alpha/\mathcal{U}$ in the language $L_{Q,P}$.
Define $R=Q_{\mathcal{M}}$.
We have $R=\prod_{\alpha\in I} R_\alpha/\mathcal{U}$, so $R$ is a ring, and $P_\mathcal{M}$ is a module over $R$.

Let $M$ be the reduct to $L_{R;P}$ of the expansion of $\mathcal{M}$ given by adding constants for the elements of $R$ (i.e. replacing the predicate of $R$ with constants).
We get that $M\vDash \mathrm{F}_{R-module}$.

Now suppose that  $V\subseteq M^s$ is a variety defined by a conjunction of polynomial equations 
$$\phi(x;b)=(f_0(x;b)=0)\land\dots\land(f_{n-1}(x;b)=0),$$
where $b$ is some tuple in $M$.
Elements of $M$ are equivalence classes of sequences in $\prod_{\alpha\in I} M_\alpha$, so we can choose a representative sequence $(b_\alpha)_{\alpha\in I}$ for $b$, where each $b_\alpha$ is a $\abs{b}$-tuple in $M_\alpha$.
Define varieties $V_\alpha=\phi(M_\alpha,b_\alpha)\subseteq M^s$ for each $\alpha\in I$.

\begin{lemma} \label{L:ultraproduct R-free}
    If $V$ is an irreducible variety, then so is $V_\alpha$ for $\mathcal{U}$-almost all $\alpha\in I$.
    If $V$ is furthermore an $R$-free variety, then $V_\alpha$ is $R_\alpha$-free for $\mathcal{U}$-almost all $\alpha\in I$.
\end{lemma}

\begin{proof}
    Suppose that $V$ is an irreducible variety.
    By \cite[Theorem 2.10(i)]{vandendries1984bounds} (see also \cite[Theorem 10.2.1]{Joh16}), there is a formula in the language of rings $\theta(y)$  with $y$ a $\abs{b}$-tuple of variables, such that for any field $K$ and $c\in K$, $K\vDash\theta(c)$ iff the variety $\phi(K,c)\subseteq K^s$ is an irreducible variety.
    We know that $\theta(b)$ holds, so $\theta(b_\alpha)$ holds for $\mathcal{U}$-almost all $\alpha\in I$.
    
    Now suppose furthermore that $V$ is $R$-free.
    For every $\alpha\in I$, we can find a field extension $N_\alpha\supseteq M_\alpha$ with $a_\alpha\in N_\alpha$ a generic point of $V_\alpha$ over $M_\alpha$.
    Let $\mathcal{N}_\alpha$ be the $L_{Q,P}$ structure on $N_\alpha$ given by interpreting $Q$ and $P$ as in $\mathcal{M}_\alpha$.
    Define $\mathcal{N}=\prod_{\alpha\in I} \mathcal{N}_\alpha/\mathcal{U}$, $\mathcal{N}$ is an $L_{Q,P}$-superstructure of $\mathcal{M}$ with $Q$ and $P$ interpreted as in $\mathcal{M}$.
    Let $a\in \mathcal{N}$ be the equivalence class of $(a_\alpha)_{\alpha\in I}$.
    
    We will show that $a$ is a generic point of $V$ over $M$.
    First of all, for all $\alpha\in I$ we have $N_\alpha\vDash\phi(a_\alpha,b_\alpha)$, so $N\vDash\phi(a,b)$.
    Now suppose $a$ satisfies some polynomial equation over $M$, $f(a;b')=0$.
    For $\mathcal{U}$-almost all $\alpha\in I$, $f(a_\alpha, b'_\alpha)=0$ holds, where $(b'_\alpha)_{\alpha\in I}$ is a representative sequence of $b'$.
    Consider only those $\alpha\in I$ where the above holds and $V_\alpha$ is irreducible.
    As $a_\alpha$ is a generic point of the irreducible variety $V_\alpha$, the polynomial $f(x;b'_\alpha)$ is in the ideal generated by $f_0(x;b_\alpha),\dots f_{n-1}(x;b_\alpha)$.
    By \cite[Theorem 1.11]{vandendries1984bounds}, there is a bound $D$, depending only on $n$ and the degrees of $f,f_0,\dots f_{n-1}$ (and not on the field or the specific polynomials), and polynomials $g_\alpha^0,\dots,g_\alpha^{n-1}$ over $M_\alpha$ with degrees at most $D$ such that 
    $$f(x;b'_\alpha)=g_\alpha^0f_0(x;b_\alpha)+\dots +g_\alpha^{n-1}f_{n-1}(x;b_\alpha).$$
    We can quantify over all polynomials of degree at most $D$, so the above property is first-order expressible in the language of rings.
    Thus there are polynomial $g^0,\dots,g^{n-1}$ over $M$ of degree at most $D$ such that 
    $$f(x;b')=g^0f_0(x;b)+\dots +g^{n-1}f_{n-1}(x;b).$$
    In particular, $f(x;b')$ is in the ideal generated by $f_0(x;b),\dots f_{n-1}(x;b)$.
    Thus $a$ is a generic point of $V$.
    
    As $V$ is an $R$-free variety, it has an $R$-linearly independent generic point over $M$, and because it is irreducible all generic points share this property (a linear dependence is in particular a polynomial equation, and all generic points of an irreducible variety solve the same polynomial equations).
    Thus $a$ is $R$-linearly independent, and because this property is first-order expressible in $L_{Q,P}$, it follows that for $\mathcal{U}$-almost all $\alpha\in I$, $a_\alpha$ is $R_\alpha$-linearly independent.
    Thus for $\mathcal{U}$-almost all $\alpha\in I$, $V_\alpha$ is an $R_\alpha$-free variety.
\end{proof}

\begin{theorem} \label{T:ultraproduct existentially closed}
    If for all $\alpha\in I$, $M_\alpha$ is an existentially closed model of $\mathrm{F}_{R_\alpha-module}$, then $M$ is an existentially closed model of $\mathrm{F}_{R-module}$.
\end{theorem}

\begin{proof}
    We check that the conditions of \cref{existentially closed} hold. We are given an $R$-free variety $V\subseteq M^s$, which, by \cref{enough to consider irreducible}, we may assume irreducible. 
    Let $0\le k\le n$ and let $(A,c)$ be a $k$-compatible pair, where $A\in \Mat_{n\times s}(R)$, $c\in M^n$. 
    We need to find $b\in V$ such that for $a=Ab+c$ we have $a_0,..,a_{k-1}\in P_M$ and $a_k,..,a_{n-1}\notin P_M$.
    
    Suppose $V$ is defined by a conjunction of polynomial equations $\phi(x;d)$.
    Let $(d^\alpha)_{\alpha\in I}$ be a representative sequence of $d$, and define $V_\alpha=\phi(M_\alpha;d^\alpha)\subseteq M_\alpha^s$.
    By \cref{L:ultraproduct R-free}, $V_\alpha$ is an $R$-free variety for $\mathcal{U}$-almost all $\alpha\in I$.
    
    Let $(A^\alpha)_{\alpha\in I}$ and $(c^\alpha)_{\alpha\in I}$ be representative sequences of $A$ and $c$ respectively.
    Being a $k$-compatible pair is first-order expressible in $L_{Q,P}$, so for $\mathcal{U}$-almost all $\alpha\in I$, $(A^\alpha,c^\alpha)$ is a $k$-compatible pair.
    We assumed that the $M_\alpha$'s are existentially closed, so for $\mathcal{U}$-almost all $\alpha\in I$, by \cref{existentially closed} we can find $b^\alpha\in V_\alpha$ such that for $a^\alpha=A^\alpha b^\alpha+c^\alpha$, $a^\alpha_0,..,a^\alpha_{k-1}\in P_{M_\alpha}$ and $a^\alpha_k,..,a^\alpha_{n-1}\notin P_{M_\alpha}$.
    Let $b\in M$ be the equivalence class of $(b^\alpha)_{\alpha\in I}$ (this sequence may be defined for only $\mathcal{U}$-almost all $\alpha\in I$), and let $a=Ab+c$.
    We have $b\in V$, $a_0,..,a_{k-1}\in P_M$ and $a_k,..,a_{n-1}\notin P_M$, as needed. 
\end{proof}

\begin{question}
    In \cite[Example 5.14]{delbee2022subfield}, the theory of the pair $(M,K)$, where $M$ is an algebraically closed field and $K$ is a pseudo-finite subfield, was shown to be simple. 
    The structure $\mathcal{M}$ in the language $L_{Q,P}$ is an expansion of the pair $(M,K)$ with a predicate for a $K$-vector subspace. Note that by \cref{ec tp2}, $\mathrm{Th}(M,P_M)$ is TP$_2$, in particular  $\mathrm{Th}(\mathcal{M})$ is not simple. 
    Is it NSOP$_1$?
    
    This question is different from the result in \cref{ec nsop1}, that the class of existentially closed modules of $\mathrm{F}_{K-module}$ is NSOP$_1$, as here we have a predicate for $K$ instead of constants.
\end{question}

\subsection{Amalgamation bases}
\begin{proposition}\label{amalgamation_bases}
    The amalgamation bases of F$_{R\text{-module}}$ are precisely the models of ACF$_{R\text{-module}}$, algebraically closed fields with an $R$-submodule.
    Furthermore, they are disjoint amalgamation bases.
\end{proposition}

\begin{proof}
    Let $M$ be an algebraically closed fields with an $R$-submodule, and $f_1:M\to M_1$, $f_2:M\to M_2$ be embeddings of fields with $R$-submodules.
    There is an algebraically closed  field $N$ and field embeddings $g_1:M_1\to N$, $g_2:M_2\to N$ such that $g_1\circ f_1=g_2\circ f_2$ and $M_1\forkindep^{\ACF}_M M_2$ in $N$, where we identify the fields with their images under the embeddings.
    In particular, $M_1\cap M_2=M$, because $M$ is algebraically closed.  
    
    Give $N$ an $L_{R;P}$ structure by defining $P_N=P_{M_1}+P_{M_2}$.
    To show that $M_1$, $M_2$ are $L_{R;P}$ substructures of $N$, we need to show that $M_1\cap P_N=P_{M_1}$, and $M_2\cap P_N=P_{M_2}$ will follow from symmetry.
    By \Cref{R-modules intersection}, $M_1\cap P_N=M_1\cap (P_{M_1}+P_{M_2})=P_{M_1}+(M_1\cap P_{M_2})$, and we have $M_1\cap P_{M_2}=M_1\cap M_2\cap P_{M_2}=M\cap P_{M_2}=P_M$, so $M_1\cap P_N=P_{M_1}+P_M=P_{M_1}$. 
    Thus, $M$ is an amalgamation base, and from $M_1\cap M_2=M$ it is a disjoint amalgamation base.
    
    Suppose $M\vDash \mathrm{F}_{R\text{-module}}$ is not algebraically closed.
    There is an element $a\in \overline{M}\setminus M$. 
    Let $M_1=M_2=\overline{M}$, but define $P_{M_1}=P_M$, $P_{M_2}=P_M+\vect{a}_R$.
    Note that $P_{M_2}\cap M=P_M$,  because if $p+ra\in (P_M+\vect{a}_R)\cap M$, then $ra\in M$, so $r=0$, which implies $p+ra=p\in P_M$; thus $M\subseteq M_2$ is an $L_{R;P}$-substructure.
    Suppose we could amalgamate $M_1$, $M_2$ to a model $N\vDash \mathrm{F}_{R\text{-module}}$ by embeddings $g_1:M_1\to N$, $g_2:M_2\to N$, such that $g_1|_M=g_2|_M$.
    By changing $N$, we can assume that $g_2$ is an inclusion $M_2\subseteq N$, and in particular $a\in P_{M_2}\subseteq P_N$.
    However, $M_1=\overline{M}$, so $\textrm{Im}(g_1)=\overline{M}$ and there is some $b\in M_1$ such that $g_1(b)=a$.
    In particular, $b\in P_{M_1}=P_M$. 
    This would imply $a=g_1(b)=b\in M$, as $g_1|_M=\mathrm{id}_M$, a contradiction.
    Thus, $M$ is not an amalgamation base.
\end{proof}

\section{Classification}

\subsection{TP$_2$}
We will construct a formula that is TP$_2$ in every JEP refinement of F$_{R\text{-module}}$, as per \Cref{definition tp2 exist}. 
In particular, this will prove that for every JEP refinement $T$, $\mathcal{EC}(T)$ is \emph{not} NTP$_2$.

\begin{lemma}\label{infinite_index}
    If $M\vDash \mathrm{F}_{R\text{-module}}$ is existentially closed, then the index $[M:P_M]=\infty$.
\end{lemma}

\begin{proof}
    If $M'\supseteq M$ is a large enough field extension, then in particular $[M':P_M]=\infty$.
    Consider the $L_{R;P}$-structure on $M'$ given by $P_{M'}=P_M$, $M$ is an $L_{R;P}$-substructure of $M'$.
    The fact that $[M':P_{M'}]=\infty$ can be expressed by existential sentences ``there exist at least $n$ elements in different $P$-cosets'' for every $n$, so by existential closedness we have $[M:P_M]=\infty$.
\end{proof}

\begin{theorem} \label{ec tp2}
    Let $T$ be some JEP refinement of F$_{R\text{-module}}$. 
    The formula $\phi(x;y,z)=y\cdot x+z\in P$ has $TP_2$ with respect to $\EC(T)$.
\end{theorem}

\begin{proof}
    Take the formula $\psi(y_1,z_1;y_2,z_2)$ to be $y_1=y_2\land z_1-z_2\notin P$. 
    Let $M\vDash T$ be existentially closed such that $\abs{M}>\abs{R}+\aleph_0$, in particular it is an amalgamation base, and it is existentially closed in F$_{R\text{-module}}$. 
    The fact that $\abs{M}>\abs{R}+\aleph_0$ implies in particular that $\dim_{\fr(R)}(M)\geq\aleph_0$, so there are $ \beta_1, \beta_2,\dots\in M$ such that $1, \beta_1, \beta_2,\dots$ are $\fr(R)$-linearly independent, and in particular $R$-linearly independent. 
    By \Cref{infinite_index}, $[M:P_M]=\infty$, so there are $\gamma_1, \gamma_2,\dots\in M$ that are all in different $P_M$-cosets. 
    Take the sequence of tuples $a_{ij}=(\beta_i, \gamma_j)$. 
    Conditions (2) and (3) of \cref{definition tp2 exist} obviously hold, it remains to show (1). 
    Let $\sigma\in (\omega\setminus\set{0})^{\omega\setminus\set{0}}$, by compactness it is enough to show that for every $n>0$, $\bigwedge_{i=1}^n( \beta_i\cdot x+ \gamma_{\sigma(i)}\in P)$ is consistent.
    
    Consider the variety $V(x_0,x_1,\dots,x_n)$ given by
    \begin{align*}
        x_1&= \beta_1 x_0, \\
        &\vdots\\
        x_n&= \beta_n x_0.
    \end{align*}
    Let $b'=(b'_0, \beta_1 b'_0,\dots, \beta_n b'_0)$ be a generic point of $V$ in some field extension.
    Suppose that for $r_0,\dots,r_n\in R$ we have 
    $$r_0b'_0+r_1 \beta_1 b'_0+\dots+r_n \beta_n b'_0\in M,$$
    $$(r_0 +r_1 \beta_1+\dots+r_n \beta_n)  b'_0\in M,$$
    then $r_0+r_1 \beta_1+\dots+r_n \beta_n=0$, else we would get $b'_0\in M$.
    But $1, \beta_1,\dots, \beta_n$ are $R$-linearly independent, so we must have $r_0=\dots=r_n=0$, thus $V$ is $R$-free.
    Consider the $n\times(n+1)$ matrix
    $$ A =
    \begin{pmatrix}
        0       & 1      &  & 0 \\
        \vdots  & & \ddots &  \\
        0       & 0      &  & 1
    \end{pmatrix}$$
    and the tuple $c=(\gamma_{\sigma(1)},\dots,\gamma_{\sigma(n)})$.
    We claim that $A$ and $c$ are $n$-compatible.
    It is enough to see that the matrix $A$ is of rank $n$, so if $r_1A_1+\dots+r_nA_n=0$, then $r_1=\dots=r_n=0$.
    From \Cref{existentially closed}, it follows that there is a point $b=(b_0, \beta_1  b_0,\dots, \beta_n b_0)\in V$ such that for 
    $$Ab+c=(\beta_1 b_0+\gamma_{\sigma(1)},\dots, \beta_n b_0+\gamma_{\sigma(n)})$$
    we have $ \beta_i  b_0+\gamma_{\sigma(i)}\in P_M$, as needed.
\end{proof}

\begin{remarkcnt}
    The above implies in particular that every JEP refinement of F$_{R\text{-module}}$ is non-simple, by \cite[Proposition A.5]{Haykazyan_2021}, where the definition for simplicity in the category of existentially closed models is given in \cite[Definition A.4]{Haykazyan_2021}.
\end{remarkcnt}

\subsection{NSOP$_1$}
We will show that every for every JEP refinement $T$ of F$_{R\text{-module}}$, $\mathcal{EC}(T)$ is NSOP$_1$, using \Cref{nsop1 condition exist}.
The independence relation that we will use was defined in \cite[Definition 3.1]{d_Elb_e_2021_generic}, and called weak independence.
We will present the definition for our specific case.

\begin{definition}
    For a model $M\vDash \ACF_{R\text{-module}}$, and subsets $A,B,C\subseteq M$, say that $A$ and $B$ are \emph{weakly independent} over $C$, and denote $A\forkindep^{w}_C B$, if $A\forkindep_C^\ACF B$ and $P_M\cap(\overline{AC}+\overline{BC})=P_M\cap \overline{AC}+P_M\cap \overline{BC}$.    
\end{definition}

\begin{remarkcnt} \label{remark obvious}
    The inclusion $P_M\cap \overline{AC}+P_M\cap \overline{BC}\subseteq P_M\cap(\overline{AC}+\overline{BC})$ is always true.
\end{remarkcnt}

\begin{lemma}[$3$-amalgamation] \label{lemma_indep_3_amalgamation}
$\ACF_{R\text{-module}}$ has $3$-amalgamation, meaning any weakly independent $\mathcal{P}^-(3)$-system of $\ACF_{R\text{-module}}$ can be completed to a weakly independent $\mathcal{P}(3)$-system.
\end{lemma}

\begin{proof}
    Suppose that $M=\set{M_s}_{s\in \mathcal{P}^-(3)}$ is a weakly independent $\mathcal{P}^-(3)$-system of $\ACF_{R\text{-module}}$, and denote $P_s=P_{M_s}$.
    By \Cref{ACF n amalgamation} there is some algebraically closed field $M_3$ that completes $M$ as an independent system of ACF.
    By embedding all the system in $M_3$, we can assume that the embeddings are inclusions. 
    Define $P_3=P_{\hat{0}}+P_{\hat{1}}+P_{\hat{2}}$, where $\hat{i}=3\setminus\set{i}$, and consider $(M_3,P_3)$ as a model of $\ACF_{R\text{-module}}$.
    We need to show that $M_3$ is an $L_{R;P}$-extension of the rest of the system, that is that $P_3\cap M_{\hat{i}}=P_{\hat{i}}$.
    By symmetry it is enough to prove for $i=0$. 
    
    By \Cref{R-modules intersection}, $P_3\cap M_{\hat{0}}=(P_{\hat{0}}+P_{\hat{1}}+P_{\hat{2}})\cap M_{\hat{0}}=P_{\hat{0}}+(P_{\hat{1}}+P_{\hat{2}})\cap M_{\hat{0}}$, so it is enough to prove $(P_{\hat{1}}+P_{\hat{2}})\cap M_{\hat{0}}\subseteq P_{\hat{0}}$.
    Let $m_{\hat{0}}\in (P_{\hat{1}}+P_{\hat{2}})\cap M_{\hat{0}}$, there are $p_{\hat{1}}\in P_{\hat{1}}$, $p_{\hat{2}}\in P_{\hat{2}}$ such that $m_{\hat{0}}=p_{\hat{1}}+p_{\hat{2}}$.
    \Cref{fact shelah} applied for $t=\hat{1}$ implies that there exists $m_{\set{0}}\in M_{\set{0}}$, $m_{\set{2}}\in M_{\set{2}}$ such that $m_{\set{2}}=p_{\hat{1}}+m_{\set{0}}$, in particular $p_{\hat{1}}\in M_{\set{0}}+M_{\set{2}}$. 
    However, by weak independence in $M_{\hat{1}}$, $P_{\hat{1}}\cap (M_{\set{0}}+M_{\set{2}})=P_{\set{0}}+P_{\set{2}}$, so $p_{\hat{1}}\in P_{\set{0}}+P_{\set{2}}$.
    Similarly, by applying \Cref{fact shelah} for $t=\hat{2}$, we get $p_{\hat{2}}\in P_{\set{0}}+P_{\set{1}}$.
    Altogether, $m_{\hat{0}}=p_{\hat{1}}+p_{\hat{2}}\in P_{\set{0}}+P_{\set{1}}+P_{\set{2}}$.
    By \Cref{R-modules intersection}, $(P_{\set{0}}+P_{\set{1}}+P_{\set{2}})\cap M_{\hat{0}}=P_{\set{0}}\cap M_{\hat{0}}+(P_{\set{1}}+P_{\set{2}})$, but $P_{\set{1}}+P_{\set{2}}\subseteq P_{\hat{0}}$, so it is enough to prove $P_{\set{0}}\cap M_{\hat{0}}\subseteq P_{\hat{0}}$.
    ACF-independence of the $\mathcal{P}(3)$-system implies that $M_{\set{0}}\forkindep^\ACF_{M_\emptyset} M_{\hat{0}}$ in $M_3$, and $M_\emptyset$ is algebraically closed so $M_{\set{0}}\cap M_{\hat{0}}=M_\emptyset$.
    Thus, 
    $$P_{\set{0}}\cap M_{\hat{0}}=P_{\set{0}}\cap M_{\set{0}}\cap M_{\hat{0}}=P_{\set{0}}\cap M_\emptyset=P_\emptyset\subseteq P_{\hat{0}}.$$
    
    It remains to show that the system is weakly independent. 
    By symmetry, there are only two general cases we need to check
    \begin{enumerate}
        \item $M_{\hat{0}}\forkindep^w_{M_{\set{1}}M_{\set{2}}} M_{\hat{1}}M_{\hat{2}}$,
        \item $M_{\hat{0}}\forkindep^w_{M_\emptyset} M_{\set{0}}$.
    \end{enumerate}
    Because the system is ACF-independent, we already have $M_{\hat{0}}\forkindep^\ACF_{M_{\set{1}}M_{\set{2}}} M_{\hat{1}}M_{\hat{2}}$, $M_{\hat{0}}\forkindep^\ACF_{M_\emptyset} M_{\set{0}}$.
    For the first case, notice that $P_3\cap M_{\hat{0}}=P_{\hat{0}}$, $P_3\cap \overline{M_{\hat{1}}M_{\hat{2}}}\supseteq P_{\hat{1}}+P_{\hat{2}}$, so
    $$P_3\cap M_{\hat{0}}+P_3\cap \overline{M_{\hat{1}}M_{\hat{2}}}\supseteq P_{\hat{0}}+P_{\hat{1}}+P_{\hat{2}}=P_3\supseteq P_3\cap (M_{\hat{0}}+\overline{M_{\hat{0}}M_{\hat{1}}})$$ 
    where the other inclusion is obvious (\Cref{remark obvious}).
    For the second case, by \Cref{R-modules intersection} $P_3\cap(M_{\hat{0}}+M_{\set{0}})=(P_{\hat{0}}+P_{\hat{1}}+P_{\hat{2}})\cap (M_{\hat{0}}+M_{\set{0}})=P_{\hat{0}}+(P_{\hat{1}}+P_{\hat{2}})\cap (M_{\hat{0}}+M_{\set{0}})$, so it is enough to show $(P_{\hat{1}}+P_{\hat{2}})\cap (M_{\hat{0}}+M_{\set{0}})\subseteq P_{\hat{0}}+P_{\set{0}}$.
    Let $p_{\hat{1}}+p_{\hat{2}}\in (P_{\hat{1}}+P_{\hat{2}})\cap (M_{\hat{0}}+M_{\set{0}})$, where $p_{\hat{1}}\in P_{\hat{1}}$, $p_{\hat{2}}\in P_{\hat{2}}$. 
    We can write $p_{\hat{1}}+p_{\hat{2}}=m_{\hat{0}}+m_{\set{0}}$, where $m_{\hat{0}}\in M_{\hat{0}}$, $m_{\set{0}}\in M_{\set{0}}$.
    By \Cref{fact shelah} applied for $t=\hat{1}$, there are $m'_{\set{0}}\in M_{\set{0}}$, $m_{\set{2}}\in M_{\set{2}}$ such that $p_{\hat{1}}+m'_{\set{0}}=m_{\set{2}}+m_{\set{0}}$, so $p_{\hat{1}}\in M_{\set{0}}+M_{\set{2}}$. By weak independence in $M_{\hat{1}}$, $P_{\hat{1}}\cap (M_{\set{0}}+M_{\set{2}})=P_{\set{0}}+P_{\set{2}}$, so $p_{\hat{1}}\in P_{\set{0}}+P_{\set{2}}$. Similarly, by applying \Cref{fact shelah} for $t=\hat{2}$, $p_{\hat{2}}\in P_{\set{0}}+P_{\set{1}}$.
    Thus, $p_{\hat{0}}+p_{\hat{2}}\in P_{\set{0}}+P_{\set{1}}+P_{\set{2}}\subseteq P_{\hat{0}}+P_{\set{0}}$.
\end{proof}

\begin{question}
For which $n>3$ does $\ACF_{R\text{-module}}$ have $n$-amalgamation with weak independence?
\end{question}

\begin{lemma}\label{lemma_isolate_acfg}
    Suppose $\mathbb{M}$ is a monster model of a JEP-refinement of F$_{R\text{-module}}$.
    For a singleton $a\in \mathbb{M}$, a tuple $b\in \mathbb{M}$ and $C\subseteq \M$, if $a\in \overline{C(b)}$, then there is a formula $\phi(x,b,c)\in \text{tp}_\exists(a/Cb)$ isolating the type, in the sense that if $a'\vDash \phi(x,b,c)$ then $a'\equiv^\exists_{Cb} a$.
    
    Furthermore, we can choose $\phi(x,b, c)$ in such a way that for any $a',b'\in \M$, $\vDash \phi(a',b',c)$ implies $a'\in \overline{C(b')}$.
\end{lemma}

\begin{proof}
    We have $a\in \overline{C(b)}$, so there is some non-zero polynomial $q(x,b,c)$ with $a$ as a root, where $c\in C$.
    In particular, the formula $q(x,b,c)=0$ belongs to $\text{tp}_\exists(a/Cb)$ and has finitely many realizations. 
    Take some formula $\phi(x,b,c)\in \text{tp}_\exists(a/Cb)$ with a minimal number of realizations, a conjunction of existential formulas is existential, so $\phi(x,b,c)$ must imply every formula in $\text{tp}_\exists(a/Cb)$.
    Let $a'\vDash \phi(x,b,c)$, it follows that $a'\vDash \text{tp}_\exists(a/Cb)$, that is $\text{tp}_\exists(a'/Cb)\supseteq \text{tp}_\exists(a/Cb)$.
    On the other hand, \cref{maximal type over set} says that $\text{tp}_\exists(a/Cb)$ is a maximal existential type, so $\text{tp}_\exists(a'/Cb) = \text{tp}_\exists(a/Cb)$.
    
    For the ``furthermore'' part, we can assume that $\phi(x,y,c)\vdash q(x,y,c)=0 \land \exists x' q(x',y,c)\ne 0$, because for $y=b$ we know it is true, so we can take the conjunction of this formula with $\phi(x,y,c)$ without changing $\phi(x,b,c)$.
    Thus, if $\vDash\phi(a',b',c)$, then in particular $a'$ is a root of the non-zero polynomial $q(x,b',c)$, so $a'\in \overline{C(b')}$.
\end{proof}

\begin{theorem} \label{ec nsop1}
    Suppose $T$ is a JEP-refinement of F$_{R\text{-module}}$, then $\mathcal{EC}(T)$ has $NSOP_1$.
\end{theorem}

\begin{proof}
    We will use \Cref{nsop1 condition exist}, with the weak independence $\forkindep^w$.

    Let $\M$ be a monster model of $T$, and let $P=P_\M$.
    Invariance, symmetry and existence over models are trivial.
    For monotonicity, suppose $A,B,C,D\subseteq \M$, and $A\forkindep^w_C BD$. 
    By monotonicity of independence in ACF, we have $A\forkindep^\ACF_C B$. 
    We also get that
    \begin{align*}
        P\cap (\overline{AC}+\overline{BC})&=P\cap (\overline{AC}+\overline{BDC})\cap (\overline{AC}+\overline{BC})\\
        &= (P\cap \overline{AC}+P\cap \overline{BDC})\cap (\overline{AC}+\overline{BC})\\
        &= P\cap \overline{AC}+P\cap \overline{BDC}\cap (\overline{AC}+\overline{BC})\\
        &= P\cap \overline{AC}+P\cap (\overline{BDC}\cap \overline{AC}+\overline{BC})
    \end{align*}
    where the last two equalities follow from \Cref{R-modules intersection}, because $P\cap \overline{AC}\subseteq \overline{AC}+\overline{BC}$ and $\overline{BC}\subseteq \overline{BDC}$.
    However, $A\forkindep^\ACF_C BD$ implies that $\overline{BDC}\cap \overline{AC}=\overline{C}$, so we get $ P\cap (\overline{AC}+\overline{BC})=P\cap \overline{AC}+P\cap \overline{BC}$.
    Thus, $A\forkindep^w_C B$.
    
    \cref{equivalence 3 amalgamation} will give us the independence theorem.
    Note that to use \cref{equivalence 3 amalgamation} we need 3-amalgamation of $\EC(T)$, but \cref{lemma_indep_3_amalgamation} gives us 3-amalgamation of $\ACF_{R\text{-module}}$.
    However, a weakly independent $\mathcal{P}^-(3)$-system $\set{A_s}_{s\in \mathcal{P}^-(3)}$ of $\EC(T)$ is in particular a weakly independent $\mathcal{P}^-(3)$-system of $\ACF_{R\text{-module}}$, as existentially closed models of $T$ are also existentially closed models of F$_{R\text{-module}}$ (by the definition of JEP refinement), and in particular algebraically closed.
    A completion of the system in $\ACF_{R\text{-module}}$, $A_3\vDash \ACF_{R\text{-module}}$, can be expanded to a model of $T$, because $A_3$ is a model of $F_{R\text{-module}}\cup\text{Th}_\exists(A_\emptyset)$ which is a companion of $T$ (\cref{amalgamation base refinement}). 
    The model of $T$ can in turn be expanded to an existentially closed model, giving a completion of the system in $\EC(T)$.
    

    For strong finite character, suppose $a\not\forkindep^w_M b$, and let $A=\overline{M(a)}$, $B=\overline{M(b)}$.
    If $a\not\forkindep_M^\ACF b$, then the result follows from strong finite character in ACF. 
    Else, there is some $s\in P\cap (A+B)\setminus (P\cap A+P\cap B)$.
    There are $\alpha\in A$, $\beta\in B$ such that $s=\alpha+\beta$.
    We claim that $\beta\notin M+P\cap B$. Otherwise, there are some $m\in M$, $p\in P\cap B$ such that $\beta=m+p$, and so $s=\alpha+m+p$.
    This implies that $s-p=\alpha+m\in P\cap A$, thus $s=\alpha+m+p\in P\cap A+P\cap B$, a contradiction.

    There are formulas $\psi_\alpha(y,a,m)\in \tp_\exists(\alpha/Ma)$ and $\psi_\beta(z,b,m)\in \tp_\exists(\beta/Mb)$ isolating their respective types as in \Cref{lemma_isolate_acfg}.
    Let $\lambda(x,b,m)$ be the formula
    $$\exists y\exists z \psi_\alpha(y, x, m)\land \psi_\beta(z, b, m)\land y+z\in P,$$
    we have $\lambda(x,b,m)\in \text{tp}_\exists(a/Mb)$.
    
    Suppose that $a'\vDash\lambda(x,b,m)$, and assume towards contradiction that $a'\forkindep^w_M b$.
    Let $A'=\overline{M(a')}$, from $a'\forkindep^\ACF_M b$ we get $A'\cap B=M$.
    Let $\alpha',\beta'$ witness the existence in $\lambda(a',b,m)$, that is  $\alpha'\vDash \psi_\alpha(y, a', m)$, $\beta'\vDash \psi_\beta(z,b,m)$, and $s':=\alpha'+\beta'\in P$.
    We have $\beta'\equiv^\exists_{Mb} \beta$, in particular $\beta'\in B$ and $\beta'\notin M+P\cap B$, and by the ``furthermore'' part of \cref{lemma_isolate_acfg}, we can assume that $\alpha'\in A'$. 
    By weak independence, $P\cap (A'+B)=P\cap A'+P\cap B$, so there are $\alpha''\in P\cap A'$, $\beta''\in P\cap B$ such that $s'=\alpha''+\beta''$.
    It follows that $\alpha''-\alpha'=\beta'-\beta''\in A'\cap B=M$, but then $\beta'=\alpha''-\alpha'+\beta''\in M+(B\cap P)$, a contradiction.
\end{proof}
\begin{remarkcnt}
    In a recent paper \cite{dobrowolski2021kimindependence}, Dobrowolski and Kamsma generalized the notion of Kim-independence to positive logic, which is a more general context than the one we deal with.
They also prove that the independence relation defined on exponential fields in \cite{Haykazyan_2021} to prove NSOP$_1$ is actually Kim-independence. A similar strategy as the one in \cite[\S 10.2]{dobrowolski2021kimindependence} seems to yield that $\forkindep^w$ is Kim-independence: extension and transitivity of $\forkindep^w$ are similar to \cite[Theorem 1.4]{D_ELB_E_2021_acfg} and the fact that F$_{R\text{-module}}$ is Hausdorff follows from Theorem~\ref{amalgamation_bases}.
\end{remarkcnt}

\section{Higher amalgamation of strong independence}
In the previous section, we used the weak independence defined in \cite[Definition 3.1]{d_Elb_e_2021_generic}.
In the above cited definition, another independence called strong independence was defined. 
Strong independence is less useful for us in the study of NSOP$_1$, because the proof of strong finite character does not work for strong independence, yet it still has properties worth studying.
In this section we will prove that strong independence has $n$-amalgamation for every $n\ge 3$.

Note that in \cite{Haykazyan_2021}, Haykazyan and Kirby defined a single independence relation that had both strong finite character and $n$-amalgamation.
This does not seem to be the case in our situation.

\begin{definition}
    For a model $M\vDash \ACF_{R\text{-module}}$, and subsets $A,B,C\subseteq M$, say that $A$ and $B$ are \emph{strongly independent} over $C$, and denote $A\forkindep^{s}_C B$, if $A\forkindep_C^\ACF B$ and $P_M\cap\overline{ABC}=P_M\cap \overline{AC}+P_M\cap \overline{BC}$.    
\end{definition}

\begin{lemma}\label{properties of strong independence}
    The following are a few model theoretic properties of strong independence that we will use. 
    \begin{itemize}
        \item (Algebraicity) If $A\forkindep^s_C B$, then $\overline{AC}\forkindep^s_{\overline{C}} \overline{BC}$.
        \item (Monotonicity) If $A\forkindep^s_C BD$, then $A\forkindep^s_C B$.
    \end{itemize}
\end{lemma}

\begin{proof}
    Algebraicity is obvious from the definition, and from algebraicity of $\forkindep^\ACF$.
    For monotonicity, suppose $A\forkindep^s_C BD$, from monotonicity of $\forkindep^\ACF$ we have $A\forkindep^\ACF_C B$.
    We also get that
    \begin{align*}
        P\cap (\overline{ABC})&=P\cap \overline{ABDC}\cap \overline{ABC}\\
        &= (P\cap \overline{AC}+P\cap \overline{BDC})\cap \overline{ABC}\\
        &= P\cap \overline{AC}+P\cap \overline{BDC}\cap \overline{ABC}\\
        &= P\cap \overline{AC}+P\cap \overline{BC},
    \end{align*}
    where the second equality is from the definition of strong independence, the third equality is from \Cref{R-modules intersection}, and the last equality is from $AB\forkindep^\ACF_{BC} BDC$, which we get from base monotonicity of $\forkindep^\ACF$. 
\end{proof}

\begin{notation} \label{notation downward closed}
    A subset $I\subseteq \mathcal{P}(n)$ is called \emph{downward-closed} if $a\in I$ and $b\subseteq a$ imply $b\in I$.
    
    For a $\mathcal{P}(n)$ ($\mathcal{P}^-(n)$)-system $F$ of fields with an $R$-submodule, and $I\subseteq \mathcal{P}(n)$ ($\mathcal{P}^-(n)$) non-empty downward-closed, let
    \begin{align*}
        F_I &= \overline{\bigcup_{a\in I} F_a}\\
        P_I &= \sum_{a\in I} P_{F_a}
    \end{align*}
    Also let $F_{\subsetneq a}=F_{\mathcal{P}^-(a)}$, and the same for $P$.
\end{notation}

\begin{lemma}\label{lemma_indep_system}
    A $\mathcal{P}(n)$ ($\mathcal{P}^-(n)$)-system $M$ of $\ACF_{R\text{-module}}$ is strongly independent iff it is $\forkindep^\text{ACF}$-independent and for every $a\in\mathcal{P}(n)$ ($\mathcal{P}^-(n)$) and $I\subseteq \mathcal{P}(a)$ non-empty downward-closed (if $a\in I$ and $b\subseteq a$, then $b\in I$), we have $P_a\cap M_I=P_I$.
\end{lemma}

\begin{proof}
    For the left to right direction, suppose $M$ is strongly independent.
    The proof is by induction on $\abs{I}$.
    If $\abs{I}=1$, then we must have $I=\set{\emptyset}$, so this case is trivial as $P_a\cap M_\emptyset = P_\emptyset$.
    If $\abs{I}>1$, then take a maximal $b\in I$ and let $I'=I\setminus \set{b}$, which is also non-empty downward-closed.
    By strong independence,  $M_{b}\forkindep^s_{M_{\subsetneq b}} \bigcup_{b\not\subseteq c\subseteq a}M_c$.
    For every $c\in I'$ we have $b\not\subseteq c\subseteq a$, so by monotonicity and algebraicity (\Cref{properties of strong independence}) $M_{b}\forkindep^s_{M_{\subsetneq b}} M_{I'}$.
    It follows that
    \begin{align*}
        P_a\cap M_I&=P_a\cap\overline{M_{b} M_{I'}}\\
        &=(P_a\cap M_b)+(P_a\cap M_{I'})\\
        &=P_{b}+P_{I'}=P_I,
    \end{align*}
    where the second equality is from strong independence and the third equality is from the induction assumption.
    
    For the right to left direction, we need to prove $M_{b}\forkindep^s_{M_{\subsetneq b}} \bigcup_{b\not\subseteq c\subseteq a}M_c$.
    Consider the downward-closed families $I'=\set{c\mid b\not\subseteq c\subseteq a}$ and $I=I'\cup \set{b}$.
    With this notation, We need to prove $M_{b}\forkindep^s_{M_{\subsetneq b}} M_{I'}$.
    By the assumption,
    \begin{align*}
      P_a\cap \overline{M_bM_{I'}} &= P_a\cap M_I = P_I= P_b+P_{I'} \\
      &= (P_a\cap M_b)+(P_a\cap M_{I'}),  
    \end{align*}
    and we already know $M_{b}\forkindep^\ACF_{M_{\subsetneq b}} M_{I'}$, so this finishes the proof.
\end{proof}

\begin{theorem}[$n$-amalgamation] \label{n-amalgamation}
    Every strongly independent $\mathcal{P}^-(n)$-system of $\ACF_{R\text{-module}}$ can be completed to a strongly independent $\mathcal{P}(n)$-system.
\end{theorem}

\begin{proof}
    Suppose $M=(M_a)_{a\in \mathcal{P}^-(n)}$ is a strongly independent $\mathcal{P}^-(n)$-system of $\ACF_{R\text{-module}}$.
    By \cref{ACF n amalgamation}, there is a field $M_n$ completing $M$ as an independent system of ACF.
    Define $P_n:=P_{\subsetneq n}=\sum_{s\subsetneq n} P_s$, we need to show that $(M_n,P_n)$ completes a strongly independent system. 
    For this we will need the following claim:
    \begin{claim}
        For every $I,J\subseteq \mathcal{P}(n)$ non-empty downward-closed, $M_I\cap P_J\subseteq P_I$.
    \end{claim}
    Suppose we proved this claim.
    For every $a\in \mathcal{P}(n)$, if we take $I=\mathcal{P}(a)$ and $J=\mathcal{P}(n)$, then we will get $M_a\cap P_n\subseteq P_a$, and the other inclusion is obvious, so $(M_n, P_n)$ completes a system of $\ACF_{R\text{-module}}$. 
    Taking $J=\mathcal{P}(a)$ and $I\subseteq \mathcal{P}(a)$, we'll get $M_I\cap P_a\subseteq P_I$, and again the other inclusion is obvious, so by \cref{lemma_indep_system} the system is strongly independent. 
    All that remains is proving the claim.
    
    We will prove the claim by induction on $\abs{IJ}$. 
    The base case is $\abs{IJ}=1$, which must mean $I=J=\set{\emptyset}$, which is trivial. 
    In the general case, first notice that if $J=\mathcal{P}(n)$, then $P_J=P_n=P_{\mathcal{P}^-(n)}$, so without loss of generality we can assume $J\subseteq\mathcal{P}^-(n)$.
    If $J\subseteq I$, then it is also trivial, else take some maximal $c\in J$ such that $c\notin I$, and consider $J'=J\setminus\set{c}$, which is also non-empty downward-closed.
    
    We have $P_J=P_c+P_{J'}$, so we need to prove that $M_I\cap (P_c+P_{J'})\subseteq P_I$. 
    Suppose $p_c+p_{J'}\in M_I\cap (P_c+P_{J'})$ for $p_c\in P_c$, $p_{J'}\in P_{J'}$.
    In particular, $p_c\in M_{IJ'}$.
    Remember that $M_{IJ'} = \overline{\bigcup_{a\in IJ'} M_a}$, so there is a tuple $m\in \bigcup_{a\in IJ'} M_a$ such that $q(p_c,m)=0$
    for some non-zero polynomial $q(x,m)$.
    By \cref{fact shelah}, there is a tuple $m'\in \bigcup_{a\in IJ'} M_{a\cap c}$ such that $q(p_c,m')=0$ and $q(x,m')$ is a non-zero polynomial.
    Let $K=\set{a\in IJ'\mid a\subseteq c}=\set{a\cap c\mid a\in IJ'}$, we get that $p_c\in M_K$.
    By \Cref{lemma_indep_system}, $P_c\cap M_K =P_K\subseteq P_{IJ'}$, so $p_c\in P_{IJ'}$. 
    Also, $p_{J'}\in P_{J'}\subseteq P_{IJ'}$, so $p_c+p_{J'}\in P_{IJ'}$. 
    We know that $c\notin IJ'$, so in particular $\abs{IJ'}<\abs{IJ}$, and by the induction hypothesis $M_I\cap P_{IJ'}\subseteq P_I$.
    Thus, $p_c+p_{J'}\in P_I$, as needed.
\end{proof}

\appendix
\section{Results on higher amalgamation} \label{Results on higher amalgamation}
Our definition of independent systems, which we borrowed from \cite{Haykazyan_2021}, is not the same as the one used by other authors, e.g. \cite{Hrushovski98simplicityand,dePiro2006group,Goodrick2013amalgamation}.
For example, in the latter definition the objects of the independent system are not necessarily models, and in the former there is no requirement about the bounded closure. 
It follows that our notion of $n$-amalgamation is different from the one used in those papers, and adapting results from one definition to another is not trivial.
In this appendix we prove well known results about higher amalgamation, using our definition.

\iftoggle{THESIS}{
\section{Higher amalgamation of ACF}
}{
\subsection{Higher amalgamation of ACF}
}
Under the common definition, ACF has $n$-amalgamation for every $n$. 
More generally, \cite[Proposition 1.6]{dePiro2006group} proves that any stable theory has $n$-amalgamation over a model for all $n$.
In this section we prove that ACF has $n$-amalgamation per our definition.

\iftoggle{THESIS}{
First, recall that for fields $A,B,C$ such that $C\subseteq A\cap B$, we say that \emph{$A$ is linearly disjoint from $B$ over $C$}, and denote $A\forkindep^l_C B$, if whenever $a_0,\dots ,a_{n-1}\in A$ are linearly independent over $C$ they are also linearly independent over $B$.
Equivalently, $A$ is linearly disjoint from $B$ over $C$ iff the canonical map $A\otimes B\to A[B]$ is an isomorphism. 
In particular, if $A\forkindep^l_C B$ and we have maps $f:A\to K$ and $g:B\to K$ (for some field $K$) such that $f|_C=g|_C$, then we can jointly extend them to a map $A.B\to K$.
For more information, see \cref{subfield preliminaries} or \cite[\S III.1.a]{lang2019introduction}.
}{
First, recall that for fields $A,B,C$ that are subfields of a common field, such that $C\subseteq A\cap B$, we say that \emph{$A$ is linearly disjoint from $B$ over $C$}, and denote $A\forkindep^l_C B$, if whenever $a_0,\dots ,a_{n-1}\in A$ are linearly independent over $C$ they are also linearly independent over $B$.
Equivalently, $A$ is linearly disjoint from $B$ over $C$ iff the canonical map $A\otimes B\to A[B]$ is an isomorphism. 
In particular, if $A\forkindep^l_C B$ and we have maps $f:A\to K$ and $g:B\to K$ (for some field $K$) such that $f|_C=g|_C$, then we can jointly extend them to a map $A.B\to K$.
For more information, see \cite[\S III.1.a]{lang2019introduction}.
}

\begin{lemma} \label{independent system linearly disjoint}
Let $F=\set{F_a}_{a\in \mathcal{P}(n)}$ be an independent $\mathcal{P}(n)$-system of ACF, where all embeddings are subset-inclusions.
Suppose $a,b_0,\dots,b_{k-1}\subseteq n$, then 
$$F_a\forkindep^l_{F_{a\cap b_0}\dots F_{a\cap b_{k-1}}}F_{b_0}\dots F_{b_{k-1}}.$$
\end{lemma}
\begin{proof}
    Suppose $\sum_i \alpha_i\beta_i=0$ for $\alpha_i\in F_a$ and $\beta_i\in F_{b_0}\dots F_{b_{k-1}}$. 
    We can write $\beta_i=q_i(\beta_{i,0},\dots,\beta_{i,k-1})$ for $\beta_{i,j}\in F_{b_j}$ and $q_i$ a rational function.
    By \cref{fact shelah}, there exist $\gamma_{i,j}\in F_{a\cap b_j}$ such that $\sum_i\alpha_iq_i(\gamma_{i,0},\dots,\gamma_{i,k-1})=0$.
    Denote 
    $$\gamma_i=q_i(\gamma_{i,0},\dots,\gamma_{i,k-1})\in F_{a\cap b_0}\dots F_{a\cap b_{k-1}},$$
    we have $\sum_i \alpha_i\gamma_i=0$ as needed.
\end{proof}

\begin{lemma} \label{extend map independent system}
     Let $F=\set{F_a}_{a\in \mathcal{P}(n)}$ ($n>0$) be an independent $\mathcal{P}(n)$-system of ACF, where all embeddings are subset-inclusions.
     Suppose $K$ is another field, and for every $a\subsetneq n$ there is an embedding $\tau_a:F_a\to K$, such that $\tau_a|_{F_b}=\tau_b$ for $b\subseteq a\subsetneq n$.
     Furthermore, suppose that $T$ is a transcendence basis of $F_n$ over $\bigcup_{a\subsetneq n} F_a$ and that $S\subseteq K$ is algebraically independent over $\bigcup_{a\subsetneq n} \tau_a(F_a)$ with $\abs{S}=\abs{T}$. 
     Then there exists an embedding $\tau_n:F_n\to K$ such that $\tau_n|_{F_a}=\tau_a$ for $a\subseteq n$ and $\tau_n(T)=S$.
\end{lemma}

\begin{proof}
    For $i,j<n$, denote $\widehat{i}=n\setminus\set{i}$ and $\widehat{i,j}=n\setminus\set{i,j}$.
    We will build by induction maps $\sigma_m:F_{\widehat{0}}\dots F_{\widehat{m-1}}\to K$ such that $\sigma_m|_{F_{\widehat{i}}}=\tau_{\widehat{i}}$ for $i<m\le n$.
    For $m=1$, set $\sigma_1$ to be $\tau_{\widehat{0}}$.
    Suppose we defined $\sigma_m$, by \cref{independent system linearly disjoint}
    $$F_{\widehat{m}}\forkindep_{F_{\widehat{0,m}}\dots F_{\widehat{m-1,m}}}^l F_{\widehat{0}}\dots F_{\widehat{m-1}}.$$
    Furthermore, for every $i<m$ $$\tau_{\widehat{m}}|_{F_{\widehat{i,m}}} =\tau_{\widehat{i,m}}=\tau_{\widehat{i}}|_{F_{\widehat{i,m}}}=\sigma_m|_{F_{\widehat{i,m}}},$$
    so $\tau_{\widehat{m}}$ and $\sigma_m$ coincide on the base of the independence.
    Thus, there exists a map $\sigma_{m+1}:F_{\widehat{0}}\dots F_{\widehat{m}}\to K$ such that $\sigma_{m+1}|_{F_{\widehat{0}}\dots F_{\widehat{m-1}}}=\sigma_m$ and $\sigma_{m+1}|_{F_{\widehat{m}}}=\tau_{\widehat{m}}$.
    
    Once we built $\sigma_m$ for every $1\le m\le n$, extend $\sigma_n:F_{\widehat{0}}\dots F_{\widehat{n-1}}\to K$ to an embedding $\tau_n:F_n\to K$ by mapping $T$ to $S$ and extending to the algebraic closure.
\end{proof}

\begin{proposition} \label{ACF n amalgamation}
    ACF has $n$-amalgamation for every $n$, with respect to non-forking independence.
\end{proposition}

\begin{proof}
    Let $F=\set{F_a}_{a\in \mathcal{P}^-(n)}$ be an independent $\mathcal{P}^-(n)$-system of ACF with embeddings $\tau_{b,a}:F_b\to F_a$ for $b\subseteq a$.
    For every $\emptyset\subsetneq a\subsetneq n$, let $T_a$ be a transcendence basis of $F_a$ over $\bigcup_{b\subsetneq a} \tau_{b,a}(F_b)$.
    By induction on $\abs{a}$, it follows that
    $$F_a=\overline{\tau_{\emptyset,a}(F_\emptyset)(\bigcup_{\emptyset\subsetneq b\subseteq a} \tau_{b,a}(T_b))}.$$
    
    Let $F_n$ be some algebraically closed field extension of $F_\emptyset$, with a large enough transcendence degree over $F_\emptyset$.
    Let $\set{S_a}_{\emptyset\subsetneq a\subsetneq n}$ be some disjoint family of subsets of $F_n$ such that $\abs{S_a}=\abs{T_a}$ and $\bigcup_{\emptyset\subsetneq a\subsetneq n} S_a$ is algebraically independent over $F_\emptyset$. 
    We will extend $F$ to a $\mathcal{P}(n)$-system by defining embeddings $\tau_{a,n}:F_a\to F_n$ for all $a\subsetneq n$. 
    The embeddings $\tau_{a,n}$ will be built by induction on $\abs{a}$.
    
    For $a=\emptyset$, define  $\tau_{\emptyset,n}:F_\emptyset\to F_n$ to be the inclusion map.
    For $a\ne \emptyset$, suppose we built $\tau_{b,n}$ for every $b\subsetneq a$.
    Consider $\set{\tau_{b,a}(F_b)}_{b\subseteq a}$ as an independent $\mathcal{P}(a)$-system (where the embeddings are subset-inclusions).
    By \cref{extend map independent system}, there exist an embedding $\tau_{a,n}:F_a\to F_n$ such that $\tau_{a,n}\circ \tau_{b,a}=\tau_{b,n}$ for $b\subsetneq a$ and $\tau_{a,n}(T_a)=S_a$. 
    
    This completes $F$ to a $\mathcal{P}(n)$-system, it remains to prove independence.
    Consider all $\set{F_a}_{a\subsetneq n}$ as subsets of $F_n$ by taking their image under $\tau_{a,n}$.
    Notice that by the way we defined $\tau_{a,n}$ (specifically, because $\tau_{\emptyset,n}(F_\emptyset)=F_\emptyset$ and $\tau_{a,n}(T_a)=S_a$), we have that after taking the image under $\tau_{a,n}$
    $$F_a = \overline{F_\emptyset (\bigcup_{\emptyset\subsetneq b\subseteq a}S_b)}.$$
    We need to prove that for every $a\subseteq n$
    $$F_a\forkindep_{\bigcup_{b\subsetneq a} F_b}^\ACF\bigcup_{a\not\subseteq c\subseteq n} F_c,$$
    which is the same, up to taking algebraic closures, as 
    $$F_\emptyset (\bigcup_{\emptyset\subsetneq d\subseteq a}S_d)\forkindep_{F_\emptyset (\bigcup_{\emptyset\subsetneq b\subsetneq a} S_b)}^\ACF F_\emptyset (\bigcup_{ a\not\subseteq c\subseteq n} S_c).$$
    This follows from the fact that
    $S_a$ is algebraically independent over $F_\emptyset (\bigcup_{ a\not\subseteq c\subseteq n} S_c)$.
\end{proof}

\iftoggle{THESIS}{
\section{The independence theorem}
}{
\subsection{The independence theorem}
}

It is a well known fact in the folklore that the independence theorem is equivalent to $3$-amalgamation.
In our case there are two differences, the definition of $3$-amalgamation is different and we work in the category of existentially closed models.
We reprove this equivalence in our setting.

\begin{proposition} \label{equivalence 3 amalgamation}
    Let $\M$ be a monster model of an inductive theory $T$ with JEP.
    Suppose that there is an ternary relation $\forkindep$ on subsets of $\M$ satisfying invariance, existence, monotonicity, symmetry, and extension.
    For $M\in \EC(T)$, the following are equivalent:
    \begin{enumerate}
        \item (3-amalgamation) every independent $\mathcal{P}^-(3)$-system of $\EC(T)$ over $M$ can be completed to an independent $\mathcal{P}(3)$-system of $\EC(T)$ (a system $F$ is over $M$ if $F_\emptyset=M$).
        \item (strengthened independence theorem) for tuples $c_1,c_2,b_1,b_2$ such that $c_1\forkindep_M c_2$, $b_1\forkindep_M c_1$, $b_2\forkindep_M c_2$ and $b_1\equiv^\exists_M b_2$, there exists $b$ such that $b\equiv^\exists_{Mc_1} b_1$, $b\equiv^\exists_{Mc_2} b_2$, and $b\forkindep_M c_1c_2$, $bc_1\forkindep_M c_2$, $bc_2\forkindep_M c_1$.
    \end{enumerate}
\end{proposition}

\begin{proof}
    $(1)\implies (2)$:
    We can find existentially closed models $C_1,C_2\in \EC(T)$ such that $Mc_i\subseteq C_i$ ($i=1,2$) and $C_1\forkindep_M C_2$ --- start with some $Mc_1\subseteq C_1\in \EC(T)$, and using extension and invariance move it by an automorphism fixing $Mc_1$ so that $C_1\forkindep_M c_2$, then do the same with some $Mc_2\subseteq C_2\in \EC(T)$.
    By extension, we can find $b_i'\equiv^\exists_{Mc_i} b_i$ ($i=1,2$) such that $b_i'\forkindep_M C_i$.
    Note that
    $$b_1'\equiv^\exists_M b_1 \equiv^\exists_M b_2 \equiv^\exists_M b_2'.$$
    We can find existentially closed models $Mb_i'\subseteq B_i\in \EC(T)$ such that $B_1\equiv^\exists_M B_2$ and $B_i\forkindep_M C_i$ ($i=1,2$) --- start with some $Mb_1'\subseteq B_1\in \EC(T)$, as before use extension and invariance to assume $B_1\forkindep_M C_1 $, then let $B_2$ be the image of $B_1$ under an automorphism given by $b_1'\equiv^\exists_M b_2'$, by extension and invariance we can move $B_2$ by an automorphism fixing $Mb_2'$ such that $B_2\forkindep_M C_2$. 
    
    Let $N_0, N_1, N_2\subseteq \M$ be some existentially closed models such that $C_1, C_2\subseteq N_0$, $C_1,B_1\subseteq N_1$ and $C_2,B_2\subseteq N_2$, and consider the $\mathcal{P}^-(3)$-system
    \[
    \begin{tikzcd}
    N_0                   & N_1                           & N_2                              \\
    C_1 \arrow[u] \arrow[ru] & C_2 \arrow[lu] \arrow[ru]        & B_1 \arrow[lu] \arrow[u, "\sigma"'] \\
                              & M \arrow[lu] \arrow[u] \arrow[ru] &                               
    \end{tikzcd}
    \]
    where all the arrows are inclusions, except for $\sigma$ which maps $B_1$ to $B_2$, fixing $M$.
    The above system is independent, so it can be completed to an independent $\mathcal{P}(3)$-system
    
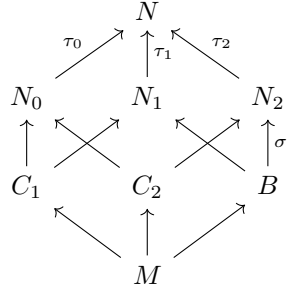
\[
    \begin{tikzcd}
                              & N                                 &                                      \\
    N_0 \arrow[ru, "\tau_0"]        & N_1 \arrow[u, "\tau_1"']      & N_2 \arrow[lu, "\tau_2"']        \\
    C_1 \arrow[u] \arrow[ru] & C_2 \arrow[lu] \arrow[ru]        & B \arrow[lu] \arrow[u, "\sigma"'] \\
                              & M \arrow[lu] \arrow[u] \arrow[ru] &                                     
    \end{tikzcd}
    \]
    We can expand $N$ to the monster $\M$, and by \cref{HK strong homogeneity} we can expand $\tau_0,\tau_1,\tau_2$ to automorphisms of $\M$.
    By applying $\tau_0^{-1}$ to $\M$, we can assume that $\tau_0$ is the identity
    
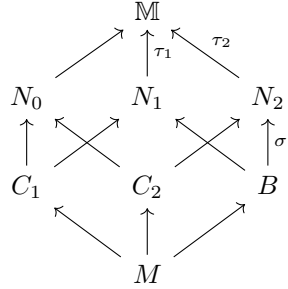
\[
    \begin{tikzcd}
                              & \M                                 &                                      \\
    N_0 \arrow[ru]        & N_1 \arrow[u, "\tau_1"']      & N_2 \arrow[lu, "\tau_2"']        \\
    C_1 \arrow[u] \arrow[ru] & C_2 \arrow[lu] \arrow[ru]        & B \arrow[lu] \arrow[u, "\sigma"'] \\
                              & M \arrow[lu] \arrow[u] \arrow[ru] &                                     
    \end{tikzcd}
    \]
    Let $b=\tau_1(b_1')=\tau_2(b_2')$.
    By following the diagram, we see that $\tau_1$ fixes $Mc_1$ and $\tau_2$ fixes $Mc_2$, so 
    $$b\equiv^\exists_{Mc_1} b_1'\equiv^\exists_{Mc_1}b_1,$$
    $$b\equiv^\exists_{Mc_2} b_2'\equiv^\exists_{Mc_2}b_2.$$
    The independences we need to show follow from the fact that the system is independent (using monotonicity).
    \\
    
    $(2)\implies (1)$:
    For the other direction, let $F$ be an independent $\mathcal{P}^-(3)$-system over $M$. 
    We will show that all but one of the embeddings can be assumed to be inclusions:
    
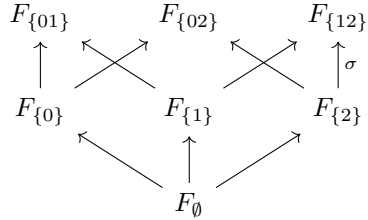
\[
    \begin{tikzcd}
    F_{\set{01}}                   & F_{\set{02}}                                      & F_{\set{12}}                              \\
    F_{\set{0}} \arrow[u] \arrow[ru] & F_{\set{1}} \arrow[lu] \arrow[ru]                   & F_{\set{2}} \arrow[lu] \arrow[u, "\sigma"'] \\
                             & F_\emptyset \arrow[lu] \arrow[u] \arrow[ru] &                                    
    \end{tikzcd}
    \]
    Start by replacing $F_\emptyset,F_{\set{0}},F_{\set{1}}$ with their images in $F_{\set{01}}$.
    Now move $F_{\set{02}}$ so that the embedding $F_{\set{0}}\to F_{\set{02}}$ would be an inclusion  (the system stays independent by invariance), and replace $F_{\set{2}}$ with its image in $F_{\set{02}}$.
    Finally, move $F_{\set{12}}$ so that the embedding $F_{\set{1}}\to F_{\set{12}}$ would be an inclusion.
    We are left only with $F_{\set{2}}\xrightarrow{\sigma} F_{\set{12}}$, which we can't assume to be an inclusion.
    
    Recall that $M=F_\emptyset$, and consider $c_1=F_{\set{0}}$, $c_2=F_{\set{1}}$, $b_1=F_{\set{2}}$ and $b_2=\sigma(b_1)$ as tuples.
    The conditions for the independence theorem hold from the independent system, so there is some $b$ satisfying $b\equiv_{Mc_1} b_1$, $b\equiv_{Mc_2} b_2$, $b\forkindep_M c_1c_2$, $bc_1\forkindep_M c_2$ and $bc_2\forkindep_M c_1$.
    There are automorphisms $\tau_1,\tau_2$ such that $\tau_1:b_1c_1\mapsto bc_1$, $\tau_2:b_2c_2\mapsto bc_2$, so the following diagram commutes:
    \[
    \begin{tikzcd}
                             & \M                                     &                                     \\
    F_{\set{01}} \arrow[ru]        & F_{\set{02}} \arrow[u, "\tau_1"']                 & F_{\set{12}} \arrow[lu, "\tau_2"']        \\
    F_{\set{0}} \arrow[u] \arrow[ru] & F_{\set{1}} \arrow[lu] \arrow[ru]                   & F_{\set{2}} \arrow[lu] \arrow[u, "\sigma"'] \\
                             & F_\emptyset \arrow[lu] \arrow[u] \arrow[ru] &                                    
    \end{tikzcd}
    \]
    
    We have $b\forkindep_M c_1c_2$, so by extension, by possibly changing $b$ and thus also changing $\tau_1,\tau_2$ while fixing $Mc_1c_2$, we have $b\forkindep_M F_{\set{01}}$, which is $F_{\set{2}}\forkindep_{F_\emptyset} F_{\set{01}}$.
    We also know $bc_1\forkindep_M c_2$, so by extension, possibly changing $\tau_1$, we get $F_{\set{02}}\forkindep_{F_\emptyset} F_{\set{1}}$. 
    Similarly, we get $F_{\set{12}}\forkindep_{F_\emptyset} F_{\set{0}}$.
    Next, by existence, $F_{\set{0}}F_{\set{1}}\forkindep_{F_{\set{0}}F_{\set{1}}} F_{\set{02}}F_{\set{12}}$, so by extension and changing $F_{\set{01}}$ (really, its embedding into $\M$) we get that $F_{\set{01}}\forkindep_{F_{\set{0}}F_{\set{1}}} F_{\set{02}}F_{\set{12}}$.
    The same can be done with $F_{\set{02}}\forkindep_{F_{\set{0}}F_{\set{2}}} F_{\set{01}}F_{\set{12}}$ and $F_{\set{12}}\forkindep_{F_{\set{1}}F_{\set{2}}} F_{\set{01}}F_{\set{02}}$.
    Notice that the automorphisms we take preserve $F_{\set{0}}F_{\set{1}}F_{\set{2}}$, so they preserve the independences already established.
    This gives us an independent $\mathcal{P}(3)$-system that completes the given independent $\mathcal{P}^-(3)$-system.
\end{proof}

\bibliographystyle{alpha}
\bibliography{library}

\end{document}